\documentclass[review]{elsarticle}
\usepackage{color}
\usepackage{lineno,hyperref}
\usepackage{amsfonts}
\usepackage{amsmath}
\usepackage{amsthm}
\usepackage{amssymb}
\usepackage{amsthm,amsmath,amssymb}
\usepackage{mathrsfs}
\usepackage{amscd,multirow}
\usepackage{ulem}

\usepackage[caption=false,font=normalsize,labelfont=sf,textfont=sf]{subfig}
\usepackage{stfloats}

\usepackage{graphicx}
\usepackage{epstopdf}
\usepackage[titletoc]{appendix}
\usepackage{subeqnarray,cases}
\usepackage[ruled]{algorithm2e}
\newtheorem{theorem}{Theorem}[section]

\newtheorem{lemma}{Lemma}[section]

\newtheorem{remark}{Remark}[section]
\newtheorem{example}{Example}[section]
\newtheorem{assumption}{Assumption}

\modulolinenumbers[5]

\begin{document}
\begin{frontmatter}
		
\title{Fast convergence rates and trajectory convergence of a Tikhonov regularized inertial primal\mbox{-}dual dynamical system with time scaling and vanishing damping\tnoteref{mytitlenote}}
		
		
\author[mymainaddress]{Ting-Ting Zhu}
\ead{zttsicuandaxue@126.com}
		
\author[mysecondaryaddress]{Rong Hu}
\ead{ronghumath@aliyun.com}
		
\author[mymainaddress]{Ya-Ping Fang\corref{mycorrespondingauthor}}
\cortext[mycorrespondingauthor]{Corresponding author}
\ead{ypfang@scu.edu.cn}
		
\address[mymainaddress]{Department of Mathematics, Sichuan University, Chengdu, Sichuan, P.R. China}
\address[mysecondaryaddress]{Department of Applied Mathematics, Chengdu University of Information Technology, Chengdu, Sichuan, P.R. China}
		
\begin{abstract}
A Tikhonov regularized inertial primal\mbox{-}dual dynamical system with time scaling and vanishing damping is proposed for solving a linearly constrained convex optimization problem in Hilbert spaces. The system under consideration consists of two coupled second order differential equations and its convergence properties depend upon the decaying speed of the product of the time scaling parameter and the Tikhonov regularization parameter (named the rescaled regularization parameter) to zero.  When the rescaled regularization parameter converges rapidly to zero, the system enjoys  fast convergence rates of the primal-dual gap, the feasibility violation, the objective residual, and the gradient norm of the objective function along the trajectory, and the weak convergence of the trajectory to a primal-dual solution of the linearly constrained convex optimization problem. When the rescaled regularization parameter converges slowly to zero, the generated primal trajectory converges strongly to the minimal norm solution of the  problem under suitable conditions. Finally,  numerical experiments are performed to illustrate the theoretical findings.
\end{abstract}
		
\end{frontmatter}
	
\section{Introduction}
\subsection{Problem statement and motivation}

Let $\mathcal{X}$ and $\mathcal{Y}$ be two real Hilbert spaces with the inner product $\langle \cdot, \cdot\rangle$ and the associated norm $\|\cdot\|$. The Cartesian product $\mathcal{X}\times\mathcal{Y}$ is endowed with the inner product and the associated norm by
$$\langle (x,y),(u,v)\rangle=\langle x,u\rangle+\langle y,v\rangle \quad\text{ and }\quad \|(x,y)\|=\sqrt{\|x\|^2+\|y\|^2}$$
for all   $(x,y),(u,v)\in \mathcal{X}\times\mathcal{Y}$.

In this paper we are concerned with  the linearly constrained convex optimization problem
\begin{equation}\label{z1}
	\min_{x\in\mathcal{X}}  \quad f(x), \quad \text{ s.t. }  \  Ax = b,
\end{equation} 
where $A: \mathcal{X}\rightarrow \mathcal{Y}$ is a continuous linear operator, $b\in\mathcal{Y}$, and $f: \mathcal{X}\rightarrow \mathbb{R}$ is a continuously differentiable convex function such that its gradient operator $\nabla f$ is Lipschitz continuous with constant $L>0$ over  $\mathcal{X}$. Problem \eqref{z1} underlies many important applications arising image recovery, machine learning, network optimization, and the energy dispatch of power grids. See e.g. \cite{GoldsteinTandDonoghue(2014),ZLinandLiHandFang(2020),ZengXLandLeiJLandChenJ(2022), PYiandHongYandLiu(2015)}. 

The Lagrangian function $\mathcal{L}$ of Problem \eqref{z1} is defined by
\begin{eqnarray*}
	\mathcal{L}(x,\lambda)=f(x)+\langle \lambda, Ax-b\rangle
\end{eqnarray*}
and the augmented Lagrangian function $\mathcal{L}_{\rho}$ with  penalty parameter $\rho\ge 0$  is defined  by
\begin{eqnarray*}\label{zz1}
	\mathcal{L}_{\rho}(x,\lambda)=\mathcal{L}(x,\lambda)+\frac{\rho}{2}\|Ax-b\|^2
	=f(x)+\langle \lambda, Ax-b\rangle +\frac{\rho}{2}\|Ax-b\|^2.
\end{eqnarray*}
 The  dual problem  and  the augmented Lagrangian dual problem of Problem \eqref{z1} are formulated respectively as
\begin{eqnarray}\label{dudu5}
	\max_{\lambda\in\mathcal{Y}}d(\lambda)
\end{eqnarray}
and 
\begin{equation}\label{dp-fyp}
	\max_{\lambda\in\mathcal{Y}}d_{\rho}(\lambda),
\end{equation}
where 
$$d(\lambda)=\min_{x\in\mathcal{X}}\mathcal{L}(x,\lambda) \quad \text{ and }\quad
d_{\rho}(\lambda)=\min_{x\in\mathcal{X}}\mathcal{L}_{\rho}(x,\lambda).$$ 

It is well-known that  $\mathcal{L}(x,\lambda)$ and $\mathcal{L}_{\rho}(x,\lambda)$ have  a same saddle point set $\Omega$, and 
\begin{eqnarray}\label{zcs1}
	(x^*, \lambda^*)\in\Omega \Leftrightarrow \left\{\begin{aligned}&\nabla f(x^*)+A^T \lambda^*&=0,\\ &Ax^*-b&=0.\end{aligned} \right.
\end{eqnarray}
If $(x^*, \lambda^*)\in\Omega$, then $x^*$ is  a solution of Problem \eqref{z1} and  $\lambda^*\in\mathcal{Y}$ is a solution of Problem \eqref{dudu5}. So, a pair $(x^*, \lambda^*)\in\Omega$ is also called a primal-dual solution of Problem \eqref{z1}.
If $\lambda^*\in\mathcal{Y}$  is a solution of Problem \eqref{dp-fyp}, according to \cite[Chapter III: Remark 2.5]{FortinandGlowinski(1983)}) and \cite[Section 2]{zhuhufang1}, it holds 
\begin{eqnarray}\label{zttonefang11}
	\mathcal{D}(\lambda^*):=\arg\min_{x\in\mathcal{X}}\mathcal{L}_{\rho}(x,\lambda^*)\subseteq S,
\end{eqnarray}
where $\rho>0$. 

If $A=0$ and $b=0$, then  Problem \eqref{z1} becomes the unconstrained convex optimization problem
\begin{equation}\label{zfg1}
	\min_{x\in\mathcal{X}}  \quad f(x).
\end{equation} 

Various Tikhonov regularized dynamical systems have been proposed  in the literature for the unconstrained problem \eqref{zfg1}. Convergence properties of a Tikhonov regularized dynamical system depends upon the speed of convergence of the Tikhonov regularization parameter to zero. A Tikhonov regularized dynamical system enjoys convergence properties similar to the original  system ( without a Tikhonov regularization term) when the  Tikhonov regularization parameter decays rapidly to zero,  while the trajectory converges strongly to the minimal norm solution of  Problem \eqref{zfg1} when the Tikhonov regularization parameter converges slowly to zero. See, e.g., \cite{AttouchandCzarnecki(2002),AttouchZH2018,AttouchandSzilardLaszlo(2021),Laszlo-JDE,XuWen2021}. In recent years, some researchers began to focuse on the studies of inertial primal-dual dynamical systems for the linearly constrained convex optimization problem \eqref{z1}.  
In this paper, we consider the following Tikhonov regularized primal-dual dynamical system with time scaling and vanishing damping
\begin{eqnarray}\label{z2}
	\begin{cases}
		\ddot{x}(t)+\frac{\alpha}{t}\dot{x}(t)&=-\beta(t)\left(\nabla f(x(t))+A^T(\lambda(t)+\theta t\dot{\lambda}(t))+\rho A^T(Ax(t)-b)
		+\epsilon(t)x(t)\right),\\
		\ddot{\lambda}(t)+\frac{\alpha}{t}\dot{\lambda}(t)&=\beta(t)\left(A(x(t)+\theta t\dot{x}(t))-b\right),
	\end{cases}
\end{eqnarray}
where $t\geq t_0>0$, $\alpha>0$, $\theta>0$, $\rho\ge 0$, $A^T$ is the adjoint operator of $A$,  $\beta:[t_0, +\infty)\rightarrow (0,+\infty)$ denotes the scaling parameter, and $\epsilon: [t_0, +\infty)\rightarrow [0,+\infty)$  denotes the Tikhonov regularization parameter. 
Throughout this paper, we will make the following standard assumption on the parameters $\alpha,\rho,\beta, \theta$, and $\epsilon$ in System \eqref{z2}:

\begin{assumption}\label{AS-F}
	Suppose that $\epsilon: [t_0,+\infty)\rightarrow [0,+\infty)$ is a  nonincreasing continuously differentiable function satisfyig $\lim_{t\to+\infty}\epsilon(t)=0$, $\beta: [t_0,+\infty)\rightarrow (0,+\infty)$ is a continuously differentiable function, and 
	$$\alpha>1,\quad\theta\ge \frac{1}{\alpha-1}, \quad \rho\ge 0, \quad \Omega\ne\emptyset.$$
\end{assumption}

Throughout this paper, we call the product $\beta(t)\epsilon(t)$  the rescaled regularization paramter, which plays a crucial role in investigating the convergence properties of System \eqref{z2}. When the rescaled regularization paramter decays rapidly to zero, we will  prove that System \eqref{z2} exhibit fast convergence rates in the primal-dual gap, the feasibility violation,  and the objective residual along the trajectory, and that  the  trajectory converges weakly to a primal-dual solution of Problem  \eqref{z1}. These convegence properties are similar to the ones of the original system (without Tikhonov regularization term) considered in the literature. When  the rescaled regularization paramter converges slowly to zero, we will derive the strong convergence of the primal trajectory to the minimal norm solution of Problem \eqref{z1} under suitable conditions. 

\subsection{Related works}
\subsubsection{Inertial dynamical systems for the unconstrained  problem \eqref{zfg1}}

Attached to Problem \eqref{zfg1}, Polyak \cite{Polyak(1964)} investigated the heavy ball with friction system
\begin{eqnarray*}
	\text{(HBF)}\quad\quad 
	\ddot{x}(t)+\gamma \dot{x}(t)+\nabla f(x(t))=0, 
\end{eqnarray*}
where $\gamma>0$ is  a fixed damping coefficient.  To understand the acceleration of   Nesterov’s accelerated gradient algorithm \cite{Nesterov(1983), Nesterov(2013)}, Su et al. \cite{SuBoydandCandes(2016)} proposed the inertial dynamical system with vanishing damping
\begin{eqnarray*}
	\text{(AVD)}\quad\quad
	\ddot{x}(t)+\frac{\alpha}{t} \dot{x}(t)+\nabla f(x(t))=0, 
\end{eqnarray*}
where $\alpha>0$ is a constant.  Since then,  inertial dynamical systems with vanishing damping  $\frac{\alpha}{t}$ have been intensively studied in the literature.  It has been shown in \cite{Attouch2018MP, Wilson2021, ShiBMP,MayLong2015} that the  damping coefficient $\frac{\alpha}{t}$ with $\alpha\ge 3$, vanishing asymptotically, but not too quickly, plays a crucial role in establishing the fast convergence rate $f(x(t))-\min f=\mathcal{O}\left(\frac{1}{t^2}\right)$. On the other hand,  time rescaling technique has been applied to improve the convergence rates of inertial dynamical systems. By introducing a time scalling coefficient $\beta(t)$ to \text{(AVD)}, Attouch et al. \cite{AttouchCRF2019} proved the inertial dynamical sysytem with time scaling
\begin{eqnarray*}
	\ddot{x}(t)+\frac{\alpha}{t} \dot{x}(t)+\beta(t)\nabla f(x(t))=0
\end{eqnarray*}
with $\alpha\ge 3$ exhibits an improved convergence rate $f(x(t))-\min f=\mathcal{O}\left(\frac{1}{t^2\beta(t)}\right)$ under the scaling condition $t\dot{\beta}(t)\leq(\alpha-3)\beta(t)$. More results on inertial dynamical systems with scaling coefficients can be found in \cite{Wilson2021,WibisonoWJA2016,FazlyabKPA2017,AttouchCRFC2019}.

In past decades,  Tikhonov regularized inertial dynamical systems for  Problem \eqref{zfg1} have been intensively investigated.  Attouch and Czarnecki \cite{AttouchandCzarnecki(2002)} considered the Tikhonov regularized version of the heavy ball with friction system $\text{(HBF)}$
\begin{eqnarray*}
	\text{(HBF)}_\epsilon \quad\quad \ddot{x}(t)+\gamma \dot{x}(t)+\nabla f(x(t))+\epsilon(t)x(t)=0,  
\end{eqnarray*}
where $\epsilon:[0,+\infty)\to[0,+\infty)$ is the Tikhonov regularization parameter. When $\int_{0}^{+\infty}\epsilon(t)dt=+\infty$ which reflects the case the parameter $\epsilon(t)$ vanishes slowly, the trajectory generated by $\text{(HBF)}_\epsilon$ converges strongly to the minimal norm solution of Problem \eqref{zfg1}. By introducing a Tikhonov regularization term to  $\text{(AVD)}$,  Attouch et al. \cite{AttouchZH2018} developed the following Tikhonov regularized dynamical system 
\begin{eqnarray*}
	\text{(AVD)}_{\epsilon}\quad\quad
	\ddot{x}(t)+\frac{\alpha}{t} \dot{x}(t)+\nabla f(x(t))+\epsilon(t)x(t)=0.
\end{eqnarray*} 
When $\epsilon(t)$ decreases rapidly to zero,  $\text{(AVD)}_{\epsilon}$ enjoys  fast convergence properties similar to $\text{(AVD)}$.  When $\epsilon(t)$ decays slowly to zero, the trajectory generated by $\text{(AVD)}_{\epsilon}$  converges strongly to the minimal norm solution of  Problem \eqref{zfg1}. For more results about Tikhonov regularized inertial dynamical systems for Problem \eqref{zfg1}, we refer the reader to \cite{AttouchZH2018,AttouchandSzilardLaszlo(2021),Laszlo-JDE,XuWen2021}.

\subsubsection{Inertial primal-dual dynamical systems for the linearly constrained problem \eqref{z1}}

To extend  the work of  Su et al. \cite{SuBoydandCandes(2016)} to  the linearly constrained problem \eqref{z1},  Zeng et al. \cite{ZengXLandLeiJLandChenJ(2022)} introduced  the inertial primal-dual  dynamical system with vanishing damping  
\begin{eqnarray*}
	\text{(Z-AVD)}\quad
	\begin{cases}
		\ddot{x}(t)+\frac{\alpha}{t}\dot{x}(t)&=-\nabla f(x(t))-A^T(\lambda(t)+\theta t\dot{\lambda}(t))-\rho A^T(Ax(t)-b),\\
		\ddot{\lambda}(t)+\frac{\alpha}{t}\dot{\lambda}(t)&=A(x(t)+\theta t\dot{x}(t))-b,
	\end{cases}
\end{eqnarray*}
where $t\ge t_0>0$, $\alpha>0$, $\rho\ge 0$, and $\theta>0$, and  proved  the  $\mathcal{O}\left(\frac{1}{t^2}\right)$ convergence rate of primal-dual gap and the  $\mathcal{O}\left(\frac{1}{t}\right)$ convergence rate of feasibility violation along the generated trajectory when $\alpha\ge 3$. Let us mention that  the original system considered  by Zeng et al. \cite{ZengXLandLeiJLandChenJ(2022)} is  \text{(Z-AVD)} with $\rho=1$.  He et al. \cite{HehufangSIAM2021} and Attouch et al. \cite{AttouchADMM(2022)} extended the work of Zeng et al. \cite{ZengXLandLeiJLandChenJ(2022)} 
by investigating inertial primal-dual dynamical systems with  general damping coefficients and time scaling coefficients for Problem \eqref{z1} with a separable structure.   Bot and Nguyen \cite{BNguyen2022} improved the results of Zeng et al. \cite{ZengXLandLeiJLandChenJ(2022)}
by establshing the $\mathcal{O}\left(\frac{1}{t^2}\right)$ fast convergence rate of the primal-dual gap, the objective residual and the feasibility violation along the trajectory of $\text{(Z-AVD)}$ with $\alpha\ge 3$. Further,  Bot and Nguyen \cite{BNguyen2022} first showed that the primal-dual trajectory of $\text{(Z-AVD)}$ with $\alpha>3$ converges weakly to a primal-dual optimal solution of Problem \eqref{z1}. This result was not addressed in \cite{ZengXLandLeiJLandChenJ(2022),HehufangSIAM2021, AttouchADMM(2022)}. By introducing a scaling coefficient $\beta(t)$ to $\text{(Z-AVD)}$, Hulett and Nguyen \cite{HulettNeuyen2023} considered the  inertial  primal-dual dynamical system with time scaling
\begin{eqnarray*}\label{beta2}
	\text{(HN-AVD)}\quad
	\begin{cases}
		\ddot{x}(t)+\frac{\alpha}{t}\dot{x}(t)+\beta(t)\left(\nabla f(x(t))+A^T(\lambda(t)+\theta t\dot{\lambda}(t))+\rho A^T(Ax(t)-b)\right)
		&=0,\\
		\ddot{\lambda}(t)+\frac{\alpha}{t}\dot{\lambda}(t)-\beta(t)\left(A(x(t)+\theta t\dot{x}(t))-b\right)&=0.
	\end{cases}
\end{eqnarray*}
Under the scaling condition
\begin{eqnarray}\label{S-C}
	t\dot{\beta}(t)\leq\frac{1-2\theta}{\theta}\beta(t),\qquad \forall t\ge t_0.
	\end{eqnarray}
they showed  that $\text{(HN-AVD)}$ with $\alpha\ge 3$ owns an improved  convergence rate $\mathcal{O}\left(\frac{1}{t^2\beta(t)}\right)$ of  the primal-dual gap, the objective residual and the feasibility violation along the trajectory. Under the strict scaling condition
\begin{eqnarray}\label{SS-C}
	\sup_{t\geq t_0}\frac{t\dot{\beta}(t)}{\beta(t)}<\frac{1-2\theta}{\theta},
	\end{eqnarray}
Hulett and Nguyen \cite{HulettNeuyen2023} further proved that the trajectory generated by $\text{(HN-AVD)}$ with $\alpha>3$ converges weakly to a primal-dual  solution of Problem \eqref{z1}. 
He et al. \cite{HeHFiietal(2022)} proposed a  second order plus first order  inertial primal-dual dynamical systems with time scaling and vanishing damping for Problem \eqref{z1}. Zhu et al. \cite{zhuhufang1}  proposed a Tikhonov regularized second order plus first order primal-dual dynamical system 
\begin{eqnarray*}
	\text{(ZHF-AVD)}\quad
	\begin{cases}
		\ddot{x}(t)+\frac{\alpha}{t}\dot{x}(t)&=-\nabla f(x(t))-A^T\lambda(t)-\rho A^T(Ax(t)-b)
		-\epsilon(t)x(t),\\
		\dot{\lambda}(t)&=t(A(x(t)+\dfrac{t}{\alpha-1}\dot{x}(t))-b),
	\end{cases}
\end{eqnarray*}
where $\alpha\ge 3$, $\rho\ge 0$, and $\epsilon: [t_0, +\infty)\rightarrow [0,+\infty)$ is a Tikhonov regularization parameter. It has been shown in \cite{zhuhufang1} that \text{(ZHF-AVD)} exhibits the $\mathcal{O}\left(\frac{1}{t^2}\right)$ fast convergence rate of the primal-dual gap, the feasibility violation and the objective residual along the trajectory  when $\epsilon(t)$ decreases rapidly to zero,  while  the primal trajectory of $\text{(ZHF-AVD)}$ converges strongly to the minimal norm solution of Problem \eqref{z1}  when $\epsilon(t)$ decreases slowly to zero.  For more results on primal-dual inertial dynamical systems, we refer the reader to \cite{HeHUFang2022sf,HHF-AA, HTLF-2023,Liuxw}. 

\subsection{Organization}

In Section 2, we discuss general minimization properties of the trajectory generated by System \eqref{z2}. Section 3 is devoted to investigating fast convergence rate results of the primal-dual gap, the feasibility violation and the objective residual along the trajectory generated by System \eqref{z2}. In Section 4, we   study the weak convergence of the trajectory to a primal\mbox{-}dual  solution of Problem \eqref{z1}. In Section 5, we prove strong convergence of the primal trajectory  to the minimal norm solution of Problem \eqref{z1}. In Section 5, we perform numerical experiments to illustrate the efficiency of System \eqref{z2}. Finally, we list some lemmas  in Appendix.

\section{Minimization properties}

In this section, we  analyze general minimization properties of the trajectory  generated by System \eqref{z2} under the scaling condition \eqref{S-C} and 
\begin{equation}\label{intbe}
\int_{t_{0}}^{+\infty}\frac{\beta(t)\epsilon(t)}{t}dt<+\infty.
\end{equation}
Besides their own interest, the results in this section will be used in Section 5 to prove strong convergence of the primal trajectory of System \eqref{z2}. Throughout this paper, we always suppose that  System \eqref{z2} admits a unique strong global solution, which can be guaranteed by standard  arguments in terms of the Cauchy-Lipschitz theorem. 

We first  give a lemma which will be used repeatedly in the convergence analysis.

\begin{lemma}\label{lemma21}
Suppose that Assumption \ref{AS-F} holds and  $(x^*,\lambda^*)\in\Omega$. Let $(x,\lambda): [t_{0},+\infty)\rightarrow \mathcal{X}\times\mathcal{Y}$ be a solution of System \eqref{z2}. Define $\mathcal{E}:[t_0,+\infty)\to [0,+\infty)$ by
\begin{eqnarray}\label{zs3}
	\mathcal{E}(t)&=&\theta^2 t^2\beta(t)\left(\mathcal{L}_{\rho}(x(t),\lambda^*)-\mathcal{L}_{\rho}(x^*,\lambda^*)+\frac{\epsilon(t)}{2}\|x(t)\|^2\right)
	+\frac{1}{2}\|x(t)-x^*+\theta t\dot{x}(t)\|^2\nonumber\\
	&&+\frac{1}{2}\|\lambda(t)-\lambda^*+\theta t\dot{\lambda}(t)\|^2+\frac{(\alpha-1)\theta-1}{2}\|x(t)-x^*\|^2\nonumber\\
	&&+\frac{(\alpha-1)\theta-1}{2}\|\lambda(t)-\lambda^*\|^2.
\end{eqnarray}
For every $t\geq t_{0}$, it holds
\begin{eqnarray*}
	\dot{\mathcal{E}}(t)&\leq&((2\theta-1)\theta t\beta(t)+\theta^2t^2\dot{\beta}(t))\left(\mathcal{L}_{\rho}(x(t),\lambda^*)-\mathcal{L}_{\rho}(x^*,\lambda^*)\right)\nonumber\\
	&&+\frac{1}{2}\left(((2\theta-1)\theta t\beta(t)+\theta^2t^2\dot{\beta}(t))\epsilon(t)+\theta^2t^2\beta(t)\dot{\epsilon}(t)\right)\|x(t)\|^2+\frac{\theta t\beta(t)\epsilon(t)}{2}\|x^*\|^2\\
	&&+(1+(1-\alpha)\theta)\theta t(\|\dot{x}(t)\|^2+\|\dot{\lambda}(t)\|^2)
	-\frac{\theta t\beta(t)\epsilon(t)}{2}\|x(t)-x^*\|^2\\
	&&-\frac{\rho\theta t\beta(t)}{2}\|Ax(t)-b\|^2, \qquad\forall t\geq t_0.
\end{eqnarray*}	
\end{lemma}

\begin{proof}
By using the definition of $\mathcal{E}(t)$, we have 
\begin{eqnarray*}
	\dot{\mathcal{E}}(t)&=&(2\theta^2t\beta(t)+\theta^2t^2\dot{\beta}(t))\left(\mathcal{L}_{\rho}(x(t),\lambda^*)-\mathcal{L}_{\rho}(x^*,\lambda^*)+\frac{\epsilon(t)}{2}\|x(t)\|^2\right)\nonumber\\
	&&+\theta^2 t^2\beta(t)\left(\langle \nabla_{x}\mathcal{L}_{\rho}(x(t),\lambda^*)+\epsilon(t)x(t), \dot{x}(t)\rangle+\frac{\dot{\epsilon}(t)}{2}\|x(t)\|^2\right)\nonumber\\
	&&+\langle x(t)-x^*+\theta t\dot{x}(t), (1+\theta)\dot{x}(t)+\theta t\ddot{x}(t)\rangle+\langle \lambda(t)-\lambda^*+\theta t\dot{\lambda}(t), (1+\theta)\dot{\lambda}(t)+\theta t\ddot{\lambda}(t)\rangle\nonumber\\
	&&+((\alpha-1)\theta-1)\langle x(t)-x^*, \dot{x}(t)\rangle
	+((\alpha-1)\theta-1)\langle \lambda(t)-\lambda^*, \dot{\lambda}(t)\rangle.
\end{eqnarray*}
From \eqref{z2} and the definition of $\mathcal{L}_{\rho}$, we get
\begin{eqnarray*}
	&&\langle x(t)-x^*+\theta t\dot{x}(t), (1+\theta)\dot{x}(t)+\theta t\ddot{x}(t)\rangle\\
	&&\quad=\langle x(t)-x^*+\theta t\dot{x}(t), (1+(1-\alpha)\theta)\dot{x}(t)-\theta t\beta(t)(\nabla_{x}\mathcal{L}_{\rho}(x(t),\lambda(t)+\theta t\dot{\lambda}(t))+\epsilon(t)x(t))\rangle\\
	&&\quad=\langle x(t)-x^*+\theta t\dot{x}(t), (1+(1-\alpha)\theta)\dot{x}(t)-\theta t\beta(t)(\nabla_{x}\mathcal{L}_{\rho}(x(t),\lambda^*)+\epsilon(t)x(t))\\
	&&\qquad-\theta t\beta(t)A^T(\lambda(t)-\lambda^*+\theta t\dot{\lambda}(t))\rangle\\
	&&\quad=(1+(1-\alpha)\theta)\langle x(t)-x^*, \dot{x}(t)\rangle+(1+(1-\alpha)\theta)\theta t\|\dot{x}(t)\|^2\\
	&&\qquad-\theta t\beta(t)\langle x(t)-x^*, \nabla_{x}\mathcal{L}_{\rho}(x(t),\lambda^*)+\epsilon(t)x(t)\rangle-\theta^2t^2\beta(t)\langle \dot{x}(t), \nabla_{x}\mathcal{L}_{\rho}(x(t),\lambda^*)+\epsilon(t)x(t)\rangle\\
	&&\qquad-\theta t\beta(t)\langle A(x(t)+\theta t\dot{x}(t))-b, \lambda(t)-\lambda^*+\theta t\dot{\lambda}(t)\rangle
\end{eqnarray*}
and 
\begin{eqnarray*}
	&&\langle \lambda(t)-\lambda^*+\theta t\dot{\lambda}(t), (1+\theta)\dot{\lambda}(t)+\theta t\ddot{\lambda}(t)\rangle\\
	&&\quad=\langle \lambda(t)-\lambda^*+\theta t\dot{\lambda}(t), (1+(1-\alpha)\theta)\dot{\lambda}(t)+\theta t\beta(t)(A(x(t)+\theta t\dot{x}(t))-b)\rangle\\
	&&\quad=(1+(1-\alpha)\theta)\langle \lambda(t)-\lambda^*, \dot{\lambda}(t)\rangle+(1+(1-\alpha)\theta)\theta t\|\dot{\lambda}(t)\|^2\\
	&&\qquad+\theta t\beta(t)\langle \lambda(t)-\lambda^*+\theta t\dot{\lambda}(t), A(x(t)+\theta t\dot{x}(t))-b)\rangle.
\end{eqnarray*}
As a consequence,
\begin{eqnarray*}
	\dot{\mathcal{E}}(t)&=&(2\theta^2t\beta(t)+\theta^2t^2\dot{\beta}(t))\left(\mathcal{L}_{\rho}(x(t),\lambda^*)-\mathcal{L}_{\rho}(x^*,\lambda^*)+\frac{\epsilon(t)}{2}\|x(t)\|^2\right)\nonumber\\
	&&+\frac{\theta^2t^2\beta(t)\dot{\epsilon}(t)}{2}\|x(t)\|^2\nonumber+(1+(1-\alpha)\theta)\theta t(\|\dot{x}(t)\|^2+\|\dot{\lambda}(t)\|^2)\\
	&&-\theta t\beta(t)\langle x(t)-x^*, \nabla_{x}\mathcal{L}_{\rho}(x(t),\lambda^*)+\epsilon(t)x(t)\rangle.
\end{eqnarray*}
Because $f(x)+\frac{\epsilon(t)}{2}\|x\|^2$ is an $\epsilon(t)$\mbox{-}strongly convex function,
\begin{eqnarray*}
	f(x^*)+\frac{\epsilon(t)}{2}\|x^*\|^2-f(x(t))-\frac{\epsilon(t)}{2}\|x(t)\|^2
	&\geq&\langle \nabla f(x(t))+\epsilon(t)x(t), x^*-x(t)\rangle\\
	&&+\frac{\epsilon(t)}{2}\|x(t)-x^*\|^2,
\end{eqnarray*}
which together the definition of  $\mathcal{L}_{\rho}$ implies 
\begin{eqnarray*}\label{zs5}
	&&-\theta t \beta(t)\langle\nabla_{x}\mathcal{L}_{\rho}(x(t),\lambda^*)+\epsilon(t)x(t), x(t)-x^*\rangle\nonumber\\
	&&\quad=\theta t\beta(t)\langle\nabla_{x}\mathcal{L}_{\rho}(x(t),\lambda^*), x^*-x(t)\rangle+\theta t\beta(t)\epsilon(t)\langle x(t), x^*-x(t)\rangle\nonumber\\
	&&\quad=\theta t\beta(t)\langle \nabla f(x(t))+A^T\lambda^*+\rho A^T(Ax(t)-b),x^*-x(t)\rangle+\theta t \beta(t)\epsilon(t)\langle x(t), x^*-x(t)\rangle\nonumber\\
	&&\quad=\theta t\beta(t)(\langle\nabla f(x(t)), x^*-x(t)\rangle-\langle \lambda^*,Ax(t)-b\rangle
	-\rho\|Ax(t)-b\|^2)
	+\theta t\beta(t)\epsilon(t)\langle x(t), x^*-x(t)\rangle\nonumber\\
	&&\quad\leq\theta t \beta(t)(\mathcal{L}_{\rho}(x^*,\lambda^*)-\mathcal{L}_{\rho}(x(t),\lambda^*))+\frac{\theta t\beta(t)\epsilon(t)}{2}(\|x^*\|^2-\|x(t)\|^2-\|x(t)-x^*\|^2)\\
	&&\qquad-\frac{\rho\theta t\beta(t)}{2}\|Ax(t)-b\|^2,
\end{eqnarray*}
where the last inequality uses \eqref{zcs1} and the convexity of $\mathcal{L}_{\rho}(\cdot,\lambda^*)$.
Then, we can obtain 
\begin{eqnarray*}
	\dot{\mathcal{E}}(t)&\leq&((2\theta-1)\theta t\beta(t)+\theta^2t^2\dot{\beta}(t))\left(\mathcal{L}_{\rho}(x(t),\lambda^*)-\mathcal{L}_{\rho}(x^*,\lambda^*)\right)\nonumber\\
	&&+\frac{1}{2}\left(((2\theta-1)\theta t\beta(t)+\theta^2t^2\dot{\beta}(t))\epsilon(t)+\theta^2t^2\beta(t)\dot{\epsilon}(t)\right)\|x(t)\|^2+\frac{\theta t\beta(t)\epsilon(t)}{2}\|x^*\|^2\\
	&&+(1+(1-\alpha)\theta)\theta t(\|\dot{x}(t)\|^2+\|\dot{\lambda}(t)\|^2)
	-\frac{\theta t\beta(t)\epsilon(t)}{2}\|x(t)-x^*\|^2\\
	&&-\frac{\rho\theta t\beta(t)}{2}\|Ax(t)-b\|^2, \qquad\forall t\geq t_0.
\end{eqnarray*}
\end{proof}
Remark that  Assumption \ref{AS-F} and \eqref{S-C} imply
\begin{eqnarray}\label{zs7}
	\begin{cases}
		1+(1-\alpha)\theta\leq0,\\
		(2\theta-1)\theta t\beta(t)+\theta^2t^2\dot{\beta}(t)\leq0, \qquad \forall t\geq t_0,\\
		((2\theta-1)\theta t\beta(t)+\theta^2t^2\dot{\beta}(t))\epsilon(t)+\theta^2t^2\beta(t)\dot{\epsilon}(t)\leq0, \quad \forall t\geq t_0.
	\end{cases}
\end{eqnarray}

With Lemma \ref{intbe} in hand, we derive  general minimization  properties of the trajectory of  System \eqref{z2}.

\begin{theorem}\label{ztt3.2}
     Suppose that Assumption \ref{AS-F} and  conditions \eqref{S-C} and \eqref{intbe} hold. Let $(x,\lambda): [t_{0},+\infty)\rightarrow \mathcal{X}\times\mathcal{Y}$ be a solution of System \eqref{z2} and $(x^*,\lambda^*)\in\Omega$.
Then the following conclusions hold:
\begin{itemize} 
  \item[(i)] $\mathcal{L}_{\rho}(x(t),\lambda^*)-\mathcal{L}_{\rho}(x^*,\lambda^*)=o\left(\frac{1}{\beta(t)}\right)$.
  \item[(ii)] Further, if $\rho>0$, then
$$\|Ax(t)-b\|^2=o\left(\frac{1}{\beta(t)}\right), \quad |f(x(t))-f(x^*)|=o\left(\frac{1}{\sqrt{\beta(t)}}\right).$$
\end{itemize}
\end{theorem}

\begin{proof}
By Lemma \ref{lemma21},
\begin{eqnarray}\label{zr2}
	\frac{\dot{\mathcal{E}}(t)}{\theta^2 t^2}&\leq&\frac{(2\theta-1)\theta t\beta(t)+\theta^2t^2\dot{\beta}(t)}{\theta^2 t^2}\left(\mathcal{L}_{\rho}(x(t),\lambda^*)-\mathcal{L}_{\rho}(x^*,\lambda^*)\right)\nonumber\\
	&&+\frac{1}{2\theta^2 t^2}\left(((2\theta-1)\theta t\beta(t)+\theta^2t^2\dot{\beta}(t))\epsilon(t)+\theta^2t^2\beta(t)\dot{\epsilon}(t)\right)\|x(t)\|^2+\frac{\beta(t)\epsilon(t)}{2\theta t}\|x^*\|^2\nonumber\\
	&&+\frac{1+(1-\alpha)\theta}{\theta t}(\|\dot{x}(t)\|^2+\|\dot{\lambda}(t)\|^2)
	-\frac{\beta(t)\epsilon(t)}{2\theta t}\|x(t)-x^*\|^2\nonumber\\
	&&-\frac{\rho\beta(t)}{2\theta t}\|Ax(t)-b\|^2, \qquad\forall t\geq t_0.
\end{eqnarray}
This together with \eqref{zs7} yields
\begin{eqnarray*}
	\dot{\mathcal{E}}(t)\leq\frac{\theta\|x^*\|^2}{2}t\beta(t)\epsilon(t), \qquad\forall t\geq t_0.
\end{eqnarray*}
Integrating it over $[t_{0},t]$, we obtain
\begin{eqnarray*}
	\mathcal{E}(t)\leq \mathcal{E}(t_0)+\frac{\theta\|x^*\|^2}{2}\int_{t_{0}}^{t}s\beta(s)\epsilon(s)ds.
\end{eqnarray*}
Dividing its both sides by $\theta^2 t^2$, we have 
\begin{eqnarray}\label{zd25}
\frac{\mathcal{E}(t)}{\theta^2 t^2}\leq \frac{\mathcal{E}(t_0)}{\theta^2 t^2}+\frac{\theta\|x^*\|^2}{2\theta^2 t^2}\int_{t_{0}}^{t}s^2\frac{\beta(s)\epsilon(s)}{s}ds.
\end{eqnarray}
Applying Lemma \ref{lemma2.2.7} with $\psi(t)=t^2$ and $\phi(t)=\frac{\beta(t)\epsilon(t)}{t}$ to \eqref{zd25}, we get
\begin{eqnarray*}
\lim_{t\rightarrow+\infty}\frac{1}{t^{2}}\int_{t_0}^{t}s^{2}\frac{\beta(s)\epsilon(s)}{s}ds=0.
\end{eqnarray*}
Since $\mathcal{E}(t)\geq0$, from \eqref{zd25} we have  
\begin{eqnarray*}
	\lim_{t\rightarrow+\infty}\frac{\mathcal{E}(t)}{\theta^2 t^2}=0,
\end{eqnarray*}
which together with \eqref{zs3}  implies
\begin{eqnarray*}
\lim_{t\rightarrow+\infty}\beta(t)\left(\mathcal{L}_{\rho}(x(t),\lambda^*)-\mathcal{L}_{\rho}(x^*,\lambda^*)\right)=0,
\end{eqnarray*} 
Using the definition of $\mathcal{L}_{\rho}$ and $\rho>0$, we have 
	$$\|Ax(t)-b\|^2=o\left(\frac{1}{\beta(t)}\right), \quad |f(x(t))-f(x^*)|=o\left(\frac{1}{\sqrt{\beta(t)}}\right).$$
\end{proof}

When  $\beta(t)$ and $\epsilon(t)$ are taken such that $\beta(t)\epsilon(t)=\frac{c}{t^r}$ with $0<r\leq2$ and $c>0$, the condition \eqref{intbe} holds.  In this case, we can  imrove Theorem \ref{ztt3.2} by etablishing some convergence rate results depending upon $r$ and $\beta(t)$.

\begin{theorem}\label{ztt3.3}
 Suppose that Assumption \ref{AS-F} with $\rho>0$ and  conditions \eqref{S-C} hold, and $\beta(t)\epsilon(t)=\frac{c}{t^r}$ with $0<r\leq2$ and $c>0$. Let $(x,\lambda): [t_{0},+\infty)\rightarrow \mathcal{X}\times\mathcal{Y}$ be a solution of  System \eqref{z2} and $(x^*,\lambda^*)\in\Omega$.
Then the following conclusions hold:
	\begin{itemize}
		\item[(i)] If $0<r<2$ and $\lim_{t\to+\infty}t^r\beta(t)=+\infty$, then the trajectory $(x(t))_{t\geq t_{0}}$ is bounded, and 
		$$\mathcal{L}(x(t),\lambda^*)-\mathcal{L}(x^*,\lambda^*)=\mathcal{O}\left(\frac{1}{t^r\beta(t)}\right), \quad |f(x(t))-f(x^*)|=\mathcal{O}\left(\frac{1}{\sqrt{t^r\beta(t)}}\right),$$
		$$\|Ax(t)-b\|=\mathcal{O}\left(\frac{1}{\sqrt{t^r\beta(t)}}\right),\quad\left\|\dot{x}(t)\right\|=\mathcal{O}\left(\frac{1}{\sqrt{t^r}}\right).$$
		\item[(ii)] If $r=2$ and $\lim_{t\to+\infty}\frac{t^2\beta(t)}{\ln t}=+\infty$,  then
		$$\mathcal{L}(x(t),\lambda^*)-\mathcal{L}(x^*,\lambda^*)=\mathcal{O}\left(\frac{\ln t}{t^2\beta(t)}\right),$$
		$$|f(x(t))-f(x^*)|=\mathcal{O}\left(\frac{\sqrt{\ln t}}{t\sqrt{\beta(t)}}\right), \quad \|Ax(t)-b\|=\mathcal{O}\left(\frac{\sqrt{\ln t}}{t\sqrt{\beta(t)}}\right).$$
	\end{itemize}
\end{theorem}
\begin{proof}
When  $\beta(t)\epsilon(t)=\frac{c}{t^r}$ with $0<r\leq2$, all the conditions of Theorem \ref{ztt3.2} are satisfied.  It follows from  \eqref{zd25} that
\begin{eqnarray}\label{14-bf}
	\frac{\mathcal{E}(t)}{\theta^2 t^2}&\leq&\frac{\mathcal{E}(t_0)}{\theta^2 t^2}+\frac{c\|x^*\|^2}{2\theta t^2}\int_{t_{0}}^{t}s^{1-r}ds.
	.\end{eqnarray}

(i) When $0<r<2$, from \eqref{14-bf} we have
\begin{eqnarray*}
	\frac{\mathcal{E}(t)}{\theta^2 t^2}\leq\frac{\mathcal{E}(t_0)}{\theta^2 t^2}-\frac{c\|x^*\|^2t_{0}^{2-r}}{2(2-r)\theta t^2}+\frac{c\|x^*\|^2}{2(2-r)\theta t^r}\leq\frac{\mathcal{E}(t_0)}{\theta^2 t^2}+\frac{c\|x^*\|^2}{2(2-r)\theta t^r},
\end{eqnarray*}
which together with \eqref{zs3} implies 
\begin{equation}\label{minp1}
0\leq\mathcal{L}_{\rho}(x(t),\lambda^*)-\mathcal{L}_{\rho}(x^*,\lambda^*)\leq\frac{\mathcal{E}(t_0)}{\theta^2 t^2\beta(t)}+\frac{c\|x^*\|^2}{2(2-r)\theta t^r\beta(t)},
\end{equation}
\begin{equation}\label{minp2}
\frac{1}{2}\|x(t)\|^2\leq\frac{\mathcal{E}(t_0)}{\theta^2 t^2\beta(t)}+\frac{\|x^*\|^2}{2(2-r)\theta},\quad\frac{1}{2}\left\|\frac{1}{\theta t}(x(t)-x^*)+\dot{x}(t)\right\|^2\leq\frac{\mathcal{E}(t_0)}{\theta^2 t^2}+\frac{c\|x^*\|^2}{2(2-r)\theta t^r}.
\end{equation}
Since $\lim_{t\rightarrow+\infty}t^r\beta(t)=+\infty$, it follows from \eqref{minp2} that the primal trajectory $(x(t))_{t\geq t_{0}}$ is bounded,  and
$$\|\dot{x}(t)\|=\mathcal{O}\left(\frac{1}{\sqrt{t^r} }\right),$$
$$\mathcal{L}_{\rho}(x(t),\lambda^*)-\mathcal{L}_{\rho}(x^*,\lambda^*)=\mathcal{O}\left(\frac{1}{t^r\beta(t)}\right).$$  
Using \eqref{minp1}, $\rho>0$ and the definition of $\mathcal{L}_{\rho}$, we obtain
$$\mathcal{L}(x(t),\lambda^*)-\mathcal{L}(x^*,\lambda^*)=\mathcal{O}\left(\frac{1}{t^r\beta(t)}\right),$$  
$$|f(x(t))-f(x^*)|=\mathcal{O}\left(\frac{1}{\sqrt{t^r\beta(t)}}\right), \quad \|Ax(t)-b\|=\mathcal{O}\left(\frac{1}{\sqrt{t^r\beta(t)}}\right).$$

(ii) When $r=2$ and $c>0$, it follows from \eqref{14-bf} that
\begin{eqnarray*}
	\frac{\mathcal{E}(t)}{\theta^2 t^2}\leq\frac{\mathcal{E}(t_0)}{\theta^2 t^2}-\frac{c\|x^*\|^2\ln t_0}{2\theta t^2}+\frac{c\|x^*\|^2\ln t}{2\theta t^2}\leq\frac{\mathcal{E}(t_0)}{\theta^2 t^2}+\frac{c\|x^*\|^2\ln t}{2\theta t^2},\forall t\ge \max\{1,t_0\}.
\end{eqnarray*}
By similar arguments as in the proof of (i), we have
$$\mathcal{L}(x(t),\lambda^*)-\mathcal{L}(x^*,\lambda^*)=\mathcal{O}\left(\frac{\ln t}{t^2\beta(t)}\right),$$
$$|f(x(t))-f(x^*)|=\mathcal{O}\left(\frac{\sqrt{\ln t}}{t\sqrt{\beta(t)}}\right), \quad \|Ax(t)-b\|=\mathcal{O}\left(\frac{\sqrt{\ln t}}{t\sqrt{\beta(t)}}\right).$$
\end{proof}

\section{Fast convergence rates}

In this section, we will investigate fast convergence rates of  the primal-dual gap, the feasibility violation and the objective residual along the trajectory generated by System \eqref{z2} under the condition
\begin{equation}\label{FastSR}
\int_{t_{0}}^{+\infty}t{\beta(t)\epsilon(t)}dt<+\infty,
\end{equation}
which reflects  the case that the rescaled regularization paramter  $\beta(t)\epsilon(t)$  converges rapidly to zero. The convergence analysis is based on the methods developed in \cite{BNguyen2022, zhuhufang1, HulettNeuyen2023}.


\begin{theorem}\label{ztt3.1}
	Suppose that Assumption \ref{AS-F} and  conditions \eqref{S-C} and \eqref{FastSR} hold,  and $\lim_{t\rightarrow+\infty}t^2\beta(t)=+\infty$.
Let $(x,\lambda): [t_{0},+\infty)\rightarrow \mathcal{X}\times\mathcal{Y}$ be a solution of System \eqref{z2} and  $(x^*,\lambda^*)\in\Omega$. Then
	the trajectory $(x(t),\lambda(t))_{t\geq t_{0}}$ is bounded and  the following statements are ture:
	\begin{itemize}
     \item[(i)]  $\rho\int_{t_{0}}^{+\infty}t\beta(t)\|Ax(t)-b\|^2dt<+\infty, \quad \int_{t_{0}}^{+\infty}t \beta(t)\epsilon(t)\|x(t)\|^2dt<+\infty$
	$$((\alpha-1)\theta-1)\int_{t_{0}}^{+\infty}t(\|\dot{x}(t)\|^2+\|\dot{\lambda}(t)\|^2)dt<+\infty.$$
	\item[(ii)] $\mathcal{L}(x(t),\lambda^*)-\mathcal{L}(x^*,\lambda^*)=\mathcal{O}\left(\frac{1}{t^2\beta(t)}\right), \quad |f(x(t))-f(x^*)|=\mathcal{O}\left(\frac{1}{t^2\beta(t)}\right)$.
	$$\|Ax(t)-b\|=\mathcal{O}\left(\frac{1}{t^2\beta(t)}\right), \;
	\|(\dot{x}(t), \dot{\lambda}(t))\|
	=\mathcal{O}\left(\frac{1}{t}\right), \; \|\nabla f(x(t))-\nabla f(x^*)\|
	=\mathcal{O}\left(\frac{1}{t\sqrt{\beta(t)}}\right).$$
	
     \item[(iii)] Further, if the scaling condition \eqref{S-C} is replaced by \eqref{SS-C}, then
	it also holds
	$$\int_{t_{0}}^{+\infty}t\beta(t)(\mathcal{L}_{\rho}(x(t),\lambda^*)-\mathcal{L}_{\rho}(x^*,\lambda^*))dt<+\infty,\quad 
	\int_{t_{0}}^{+\infty}t\beta(t)\|\nabla f(x(t))-\nabla f(x^*)\|^2dt<+\infty.$$
\end{itemize}

\end{theorem}

\begin{proof}
By Lemma \ref{lemma21}, we have 
\begin{eqnarray}\label{zs8}
	&&\mathcal{E}(t)-\int_{t_{0}}^{t}((2\theta-1)\theta s\beta(s)+\theta^2s^2\dot{\beta}(s))(\mathcal{L}_{\rho}(x(s),\lambda^*)-\mathcal{L}_{\rho}(x^*,\lambda^*))ds\nonumber\\
	&&\qquad-\frac{1}{2}\int_{t_{0}}^{t}\left(((2\theta-1)\theta s\beta(s)+\theta^2s^2\dot{\beta}(s))\epsilon(s)+\theta^2s^2\beta(s)\dot{\epsilon}(s)\right)\|x(s)\|^2ds\nonumber\\
	&&\qquad-(1+(1-\alpha)\theta)\theta\int_{t_{0}}^{t} s(\|\dot{x}(s)\|^2+\|\dot{\lambda}(s)\|^2)ds
	+\frac{\theta}{2}\int_{t_{0}}^{t}s\beta(s)\epsilon(s)\|x(s)-x^*\|^2ds\nonumber\\
	&&\qquad+\frac{\rho\theta}{2}\int_{t_{0}}^{t}s\beta(s)\|Ax(s)-b\|^2ds\leq\mathcal{E}(t_0)+\frac{\theta \|x^*\|^2}{2}\int_{t_{0}}^{t}s\beta(s)\epsilon(s)ds,\quad\forall t\geq t_{0}.
\end{eqnarray}
Since $\int_{t_{0}}^{+\infty}t\beta(t)\epsilon(t)dt<+\infty$ and  $\mathcal{L}_{\rho}(x(t),\lambda^*)-\mathcal{L}_{\rho}(x^*,\lambda^*)\geq0$,  the boundedness of  $(\mathcal{E}(t))_{t\geq t_0}$  follows from and \eqref{zs7}and \eqref{zs8}.

(i) Since  $(\mathcal{E}(t))_{t\geq t_0}$ is bounded, it follows from \eqref{zs8} and \eqref{FastSR} that

\begin{eqnarray}\label{zs12}
\int_{t_{0}}^{+\infty}((1-2\theta)\theta t\beta(t)-\theta^2t^2\dot{\beta}(t))(\mathcal{L}_{\rho}(x(t),\lambda^*)-\mathcal{L}_{\rho}(x^*,\lambda^*))dt<+\infty,
\end{eqnarray}
$$\rho\int_{t_{0}}^{+\infty}t\beta(t)\|Ax(t)-b\|^2dt<+\infty,$$
\begin{eqnarray*}
((\alpha-1)\theta-1)\int_{t_{0}}^{+\infty}t(\|\dot{x}(t)\|^2+\|\dot{\lambda}(t)\|^2)dt<+\infty,
\end{eqnarray*}
\begin{eqnarray}\label{zs14}
	\int_{t_{0}}^{+\infty}t\beta(t)\epsilon(t)\|x(t)-x^*\|^2dt<+\infty.
\end{eqnarray}
Using \eqref{zs14} and \eqref{FastSR}, we  obtain 
$$\int_{t_{0}}^{+\infty}t\beta(t)\epsilon(t)\|x(t)\|^2dt<+\infty.$$

(ii) Using \eqref{zs3} and the boundedness of  $(\mathcal{E}(t))_{t\geq t_0}$, we have 
\begin{eqnarray}\label{zs11}
	\mathcal{L}(x(t),\lambda^*)-\mathcal{L}(x^*,\lambda^*)=\mathcal{O}\left(\frac{1}{t^2\beta(t)}\right)
\end{eqnarray} 
and 
\begin{equation}\label{mix-fa}
\frac{1}{2}\|x(t)-x^*+\theta t\dot{x}(t)\|^2\leq C_1,\quad \frac{1}{2}\|\lambda(t)-\lambda^*+\theta t\dot{\lambda}(t)\|^2\leq C_2, \qquad \forall t\geq t_0.
\end{equation}
Applying Lemma \ref{lemma2.1.1} to \eqref{mix-fa}, we obtain that the trajectory $(x(t), \lambda(t))_{t\geq t_0}$ is bounded, and then
$$\|(\dot{x}(t),\dot{\lambda}(t) )\|=\mathcal{O}(\frac{1}{t}).$$

Since $f$ is convex and $\nabla f$ is $L$-Lipschitz continuous,
\begin{eqnarray*}
	f(x(t))-f(x^*)-\langle\nabla f(x^*), x(t)-x^*\rangle
	\geq\dfrac{1}{2L}\|\nabla f(x(t))-\nabla f(x^*)\|^2,
\end{eqnarray*}
which  implies
\begin{eqnarray}\label{zs13}
	\mathcal{L}(x(t),\lambda^*)-\mathcal{L}(x^*,\lambda^*)&=&f(x(t))-f(x^*)+\langle \lambda^*, Ax(t)-b\rangle \nonumber\\
	&\geq&\langle\nabla f(x^*), x(t)-x^*\rangle+\langle A^T\lambda^*, x(t)-x^*\rangle\nonumber\\
	&&+\dfrac{1}{2L}\|\nabla f(x(t))-\nabla f(x^*)\|^2\nonumber\\
	&=&\dfrac{1}{2L}\|\nabla f(x(t))-\nabla f(x^*)\|^2,
\end{eqnarray} 
where the last equality uses \eqref{zcs1}. From  \eqref{zs11} and \eqref{zs13}, we get 
$$\|\nabla f(x(t))-\nabla f(x^*)\|=\mathcal{O}\left(\frac{1}{t\sqrt{\beta(t)}}\right).$$

According to \eqref{z2}, we have 
\begin{eqnarray*}
	&&t\dot{\lambda}(t)-\lambda(t)-(t_0\dot{\lambda}(t_0)-\lambda(t_{0}))\\
	&&\quad=\int_{t_{0}}^{t}s\ddot{\lambda}(s)ds=\int_{t_{0}}^{t}(-\alpha\dot{\lambda}(s)+s\beta(s)(A(x+\theta s\dot{x}(s))-b))ds\nonumber\\
	&&\quad=-\alpha\int_{t_{0}}^{t}\dot{\lambda}(s)ds+\int_{t_{0}}^{t}s\beta(s)(Ax(s)-b)ds+\int_{t_{0}}^{t}\theta s^2\beta(s)d(Ax(s)-b)\nonumber\\
	&&\quad=-\alpha(\lambda(t)-\lambda(t_0))+\int_{t_{0}}^{t}s\beta(s)(Ax(s)-b)ds-\int_{t_{0}}^{t}(\theta s^2\dot{\beta}(s)+2\theta s\beta(s))(Ax(s)-b)ds\\
	&&\qquad+\theta t^2\beta(t)(Ax(t)-b)-\theta t_{0}^2\beta(t_{0})(Ax(t_{0})-b)\\
	&&\quad=-\alpha(\lambda(t)-\lambda(t_0))+\int_{t_{0}}^{t}((1-2\theta)s\beta(s)-\theta s^2\dot{\beta}(s))(Ax(s)-b)ds\\
	&&\qquad+\theta t^2\beta(t)(Ax(t)-b)-\theta t_{0}^2\beta(t_{0})(Ax(t_{0})-b), 
\end{eqnarray*}
which together with  the boundedness of $(x(t), \lambda(t))_{t\geq t_0}$ and 
$\|(\dot{x}(t),\dot{\lambda}(t) )\|=\mathcal{O}(\frac{1}{t})$ implies 
\begin{eqnarray}\label{zs10}
\|\theta t^2\beta(t)(Ax(t)-b)+\int_{t_{0}}^{t}\frac{((1-2\theta)s\beta(s)-\theta s^2\dot{\beta}(s))}{\theta s^2\beta(s)}\theta s^2\beta(s)(Ax(s)-b)ds\|\leq C_3,\forall t\ge t_0.
\end{eqnarray}
From \eqref{S-C}, we have $\frac{(1-2\theta)t\beta(t)-\theta t^2\dot{\beta}(t)}{\theta t^2\beta(t)}\geq0 $. 
Applying Lemma \ref{lemma2.1.2} with $g(t)=\theta t^2\beta(t)(Ax(t)-b)$ and $a(t)=\frac{(1-2\theta)t\beta(t)-\theta t^2\dot{\beta}(t)}{\theta t^2\beta(t)}$ to \eqref{zs10}, we obtain 
$$\sup_{t\geq t_{0}}\theta t^2\beta(t)\|Ax(t)-b\|<+\infty.$$
As a result,  
\begin{eqnarray*}
	\|Ax(t)-b\|=\mathcal{O}\left(\frac{1}{t^2\beta(t)}\right),
\end{eqnarray*}
which together with \eqref{zs11} yields 
\begin{eqnarray*}
	|f(x(t))-f(x^*)|=\mathcal{O}\left(\frac{1}{t^2\beta(t)}\right).
\end{eqnarray*}

(iii) By \eqref{SS-C}, there exists a constant $b>0$ small enough such that 
\begin{eqnarray}\label{zs19}
	 t\dot{\beta}(t)\leq(\frac{1-2\theta}{\theta}-b)\beta(t), \quad \forall t\geq t_0,
\end{eqnarray}
which together with \eqref{zs12} implies  
$$\int_{t_{0}}^{+\infty}t\beta(t)(\mathcal{L}_{\rho}(x(t),\lambda^*)-\mathcal{L}_{\rho}(x^*,\lambda^*))dt<+\infty.$$
By using \eqref{zs13}, we have 
\begin{eqnarray*}
\mathcal{L}_{\rho}(x(t),\lambda^*)-\mathcal{L}_{\rho}(x^*,\lambda^*)
\geq\mathcal{L}(x(t),\lambda^*)-\mathcal{L}(x^*,\lambda^*)\geq\dfrac{1}{2L}\|\nabla f(x(t))-\nabla f(x^*)\|^2.
\end{eqnarray*}
As a result,
$$\int_{t_{0}}^{+\infty}t\beta(t)\|\nabla f(x(t))-\nabla f(x^*)\|^2dt<+\infty.$$ 

\end{proof}


\begin{remark}\label{corollary: eqfangw1}
Attouch et al. \cite[Theorem 3.1]{AttouchZH2018} showed  that $\text{(AVD)}_{\epsilon}$  exhibits  fast convergence rates to similar to  $\text{(AVD)}$ when $\int_{t_0}^{+\infty}\epsilon(t)dt<+\infty$, which reflects the case that  the Tikhonov regularization parameter
$\epsilon(t)$ converges rapidly to zero. The convergen rate results of Theorem \ref{ztt3.1} are consistent with the ones of  \cite[Theorem 3.2 and Theorem 3.3]{HulettNeuyen2023} and  \cite[Theorem 3.4]{BNguyen2022}. This indicates that when  \eqref{FastSR} is satisfied, System \eqref{z2}, the Tikhonov regularized version of  $\text{ (HN-AVD)}$ enjoys  convergence rates similar  to $\text{ (HN-AVD)}$.  In this sense,  Theorem \ref{ztt3.1} can be regarded as an extension of \cite[Theorem 3.1]{AttouchZH2018}. The convergence rate results in Theorem \ref{ztt3.1} are also consistent with the ones reported in \cite[Theorem 3.2]{zhuhufang1} on $\text{(ZHF-AVD)}$.
\end{remark}

\section{Weak convergence of the trajectory}

In this section,  following the appoaches developed in \cite{BNguyen2022, HulettNeuyen2023}, we will prove the weak covergence of the trajectory generated by System  \eqref{z2} to a primal-dual solution of Problem \eqref{z1}. To do this, we need  the following  strengthened version of Assumption \ref{AS-F}.
\begin{assumption}\label{as1}
	Suppose that $\epsilon: [t_0,+\infty)\rightarrow [0,+\infty)$ is a  nonincreasing and  continuously differentiable function satisfying $\lim_{t\to+\infty}\epsilon(t)=0$, $\beta: [t_0,+\infty)\rightarrow (0,+\infty)$ is a nondecreasing and continuously differentiable  function, and
	$$\frac{1}{\alpha-1}<\theta<\frac{1}{2}, \quad \alpha>3, \quad \rho>0, \quad \Omega\ne\emptyset.$$
\end{assumption}

Let $(x,\lambda): [t_{0},+\infty)\rightarrow \mathcal{X}\times\mathcal{Y}$ be a solution of System \eqref{z2} and  $(x^*,\lambda^*)\in\Omega$. Define  $h: [t_0,+\infty)\rightarrow [0, +\infty)$ and $W: [t_0,+\infty)\rightarrow[0, +\infty)$ by 
\begin{eqnarray*}
	h(t)=\frac{1}{2}\|(x(t),\lambda(t))-(x^*,\lambda^*)\|^2
\end{eqnarray*}
and 
\begin{eqnarray}\label{zs17}
W(t)=\beta(t)\left(\mathcal{L}_{\rho}(x(t),\lambda^*)-\mathcal{L}_{\rho}(x^*,\lambda^*)+\frac{\epsilon(t)}{2}\|x(t)\|^2\right)+\frac{1}{2}\|(\dot{x}(t),\dot{\lambda}(t))\|^2.
\end{eqnarray}

\begin{lemma}\label{lemma4.1.1}
	Suppose that Assumption \ref{as1}  and \eqref{SS-C} hold. Let $(x,\lambda): [t_{0},+\infty)\rightarrow \mathcal{X}\times\mathcal{Y}$ be a solution of  System \eqref{z2} and  $(x^*,\lambda^*)\in\Omega$. Then, it holds:
	\begin{eqnarray*}
		\ddot{h}(t)+\frac{\alpha}{t}\dot{h}(t)+\theta t\dot{W}(t)&\leq&\frac{\beta(t)\epsilon(t)}{2}\|x^*\|^2
		-\dfrac{\beta(t)}{2L}\|\nabla f(x(t))-\nabla f(x^*)\|^2\\
	&&-\frac{\rho\beta(t)}{2}\|Ax(t)-b\|^2, \quad \forall t\geq t_{0}.
	\end{eqnarray*}
\end{lemma}

\begin{proof}
	By the definition of $h(t)$, we obtain
	$$\dot{h}(t)=\langle x(t)-x^*,\dot{x}(t)\rangle+\langle \lambda(t)-\lambda^*,\dot{\lambda}(t)\rangle$$
and
	$$\ddot{h}(t)=\langle x(t)-x^*,\ddot{x}(t)\rangle+\langle \lambda(t)-\lambda^*,\ddot{\lambda}(t)\rangle+\|\dot{x}(t)\|^2+\|\dot{\lambda}(t)\|^2.$$
	It follows from \eqref{z2} that
	\begin{eqnarray}\label{satf1}
		\ddot{h}(t)+\frac{\alpha}{t}\dot{h}(t)&=&\|\dot{x}(t)\|^2+\|\dot{\lambda}(t)\|^2+\langle x(t)-x^*,\ddot{x}(t)+\frac{\alpha}{t}\dot{x}(t)\rangle+\langle \lambda(t)-\lambda^*,\ddot{\lambda}(t)+\frac{\alpha}{t}\dot{\lambda}(t)\rangle\nonumber\\
		&=&\|\dot{x}(t)\|^2+\|\dot{\lambda}(t)\|^2-\beta(t)\langle x(t)-x^*, \nabla_{x}\mathcal{L}_{\rho}(x(t),\lambda^*)\rangle\nonumber\\
		&&-\beta(t)\langle x(t)-x^*, A^T(\lambda(t)-\lambda^*+\theta t\dot{\lambda}(t)\rangle-\beta(t)\epsilon(t)\langle x(t)-x^*, x(t)\rangle\nonumber\\
		&&+\beta(t)\langle \lambda(t)-\lambda^*, Ax(t)-b\rangle+\theta t\beta(t)\langle \lambda(t)-\lambda^*, A\dot{x}(t)\rangle\nonumber\\
		&=&\|\dot{x}(t)\|^2+\|\dot{\lambda}(t)\|^2-\beta(t)\langle x(t)-x^*, \nabla_{x}\mathcal{L}_{\rho}(x(t),\lambda^*)\rangle\nonumber\\
		&&-\theta t\beta(t)\langle x(t)-x^*, A^T\dot{\lambda}(t)\rangle-\beta(t)\epsilon(t)\langle x(t)-x^*, x(t)\rangle\nonumber\\
		&&+\theta t\beta(t)\langle \lambda(t)-\lambda^*, A\dot{x}(t)\rangle,
	\end{eqnarray}
	where the last equality uses \eqref{zcs1}.
	Since
	\begin{eqnarray*}
		f(x^*)-f(x(t))-\langle\nabla f(x(t)), x^*-x(t)\rangle
		\geq\dfrac{1}{2L}\|\nabla f(x(t))-\nabla f(x^*)\|^2,
	\end{eqnarray*}
	it follows that 
	\begin{eqnarray*}
		-\langle x(t)-x^*, \nabla_{x}\mathcal{L}_{\rho}(x(t),\lambda^*)\rangle &=&\langle x^*-x(t), \nabla f(x(t))\rangle-\langle \lambda^*, Ax(t)-b\rangle-\rho\|Ax(t)-b\|^2\\
		&\leq&f(x^*)-f(x(t))-\dfrac{1}{2L}\|\nabla f(x(t))-\nabla f(x^*)\|^2\\
		&&-\langle \lambda^*, Ax(t)-b\rangle-\rho\|Ax(t)-b\|^2\\
		&=&-\left(\mathcal{L}_{\rho}(x(t),\lambda^*)-\mathcal{L}_{\rho}(x^*,\lambda^*)\right)-\dfrac{1}{2L}\|\nabla f(x(t))-\nabla f(x^*)\|^2\\
		&&-\frac{\rho}{2}\|Ax(t)-b\|^2.
	\end{eqnarray*}
	This together with \eqref{satf1} yields
	\begin{eqnarray}\label{satf2}
		\ddot{h}(t)+\frac{\alpha}{t}\dot{h}(t)&\leq&\|\dot{x}(t)\|^2+\|\dot{\lambda}(t)\|^2-\beta(t)\left(\mathcal{L}_{\rho}(x(t),\lambda^*)-\mathcal{L}_{\rho}(x^*,\lambda^*)\right)\nonumber\\
		&&-\dfrac{\beta(t)}{2L}\|\nabla f(x(t))-\nabla f(x^*)\|^2-\frac{\rho\beta(t)}{2}\|Ax(t)-b\|^2\nonumber\\
		&&-\theta t\beta(t)\langle x(t)-x^*, A^T\dot{\lambda}(t)\rangle-\beta(t)\epsilon(t)\langle x(t)-x^*, x(t)\rangle\nonumber\\
		&&+\theta t\beta(t)\langle \lambda(t)-\lambda^*, A\dot{x}(t)\rangle.
	\end{eqnarray}
	According to \eqref{zs17}, we have 
	\begin{eqnarray*}
		\dot{W}(t)&=&\beta(t)\left(\langle\nabla_{x}\mathcal{L}_{\rho}(x(t),\lambda^*), \dot{x}(t)\rangle+\epsilon(t)\langle x(t),\dot{x}(t)\rangle+\frac{\dot{\epsilon}(t)}{2}\|x(t)\|^2\right)+\langle\dot{x}(t), \ddot{x}(t)\rangle\\
		&&+\langle\dot{\lambda}(t), \ddot{\lambda}(t)\rangle+\dot{\beta}(t)\left(\mathcal{L}_{\rho}(x(t),\lambda^*)-\mathcal{L}_{\rho}(x^*,\lambda^*)+\frac{\epsilon(t)}{2}\|x(t)\|^2\right)\\
		&=&\beta(t)\langle\nabla_{x}\mathcal{L}_{\rho}(x(t),\lambda^*), \dot{x}(t)\rangle+\beta(t)\epsilon(t)\langle x(t),\dot{x}(t)\rangle+\frac{\beta(t)\dot{\epsilon}(t)}{2}\|x(t)\|^2\\
		&&-\frac{\alpha}{t}\|\dot{x}(t)\|^2-\beta(t)\langle\dot{x}(t), \nabla_{x}\mathcal{L}_{\rho}(x(t),\lambda(t)+\theta t\dot{\lambda}(t))\rangle-\beta(t)\epsilon(t)\langle \dot{x}(t),x(t)\rangle\\
		&&-\frac{\alpha}{t}\|\dot{\lambda}(t)\|^2+\beta(t)\langle\dot{\lambda}(t), A(x(t)+\theta t\dot{x}(t))-b\rangle+\dot{\beta}(t)(\mathcal{L}_{\rho}(x(t),\lambda^*)-\mathcal{L}_{\rho}(x^*,\lambda^*))\\
		&&+\frac{\dot{\beta}(t)\epsilon(t)}{2}\|x(t)\|^2.
	\end{eqnarray*}
Since
	$$\nabla_{x}\mathcal{L}_{\rho}(x(t),\lambda(t)+\theta t\dot{\lambda}(t))=\nabla_{x}\mathcal{L}_{\rho}(x(t),\lambda^*)+A^T(\lambda(t)-\lambda^*+\theta t\dot{\lambda}(t)),$$  
it follows that
	\begin{eqnarray*}
		\dot{W}(t)&=&\frac{\beta(t)\dot{\epsilon}(t)}{2}\|x(t)\|^2-\frac{\alpha}{t}\|\dot{x}(t)\|^2-\frac{\alpha}{t}\|\dot{\lambda}(t)\|^2
		-\beta(t)\langle A\dot{x}(t), \lambda(t)-\lambda^*\rangle\\
		&&+\beta(t)\langle\dot{\lambda}(t), Ax(t)-b\rangle+\dot{\beta}(t)(\mathcal{L}_{\rho}(x(t),\lambda^*)-\mathcal{L}_{\rho}(x^*,\lambda^*))
		+\frac{\dot{\beta}(t)\epsilon(t)}{2}\|x(t)\|^2,
	\end{eqnarray*}
	which together with \eqref{satf2} implies
	\begin{eqnarray}\label{zs18}
		\ddot{h}(t)+\frac{\alpha}{t}\dot{h}(t)+\theta t\dot{W}(t)&\leq&(\theta t\dot{\beta}(t)-\beta(t))(\mathcal{L}_{\rho}(x(t),\lambda^*)-\mathcal{L}_{\rho}(x^*,\lambda^*))\nonumber\\
		&&+(1-\alpha\theta)(\|\dot{x}(t)\|^2+\|\dot{\lambda}(t)\|^2)+\left(\frac{\theta t\dot{\beta}(t)\epsilon(t)}{2}+\frac{\theta t\beta(t)\dot{\epsilon}(t)}{2}\right)\|x(t)\|^2\nonumber\\
		&&-\dfrac{\beta(t)}{2L}\|\nabla f(x(t))-\nabla f(x^*)\|^2-\frac{\rho\beta(t)}{2}\|Ax(t)-b\|^2\nonumber\\
		&&-\beta(t)\epsilon(t)\langle x(t)-x^*, x(t)\rangle\nonumber\\
		&\leq&(1-\alpha\theta)(\|\dot{x}(t)\|^2+\|\dot{\lambda}(t)\|^2)+\frac{1-2\theta}{2}\beta(t)\epsilon(t)\|x(t)\|^2\nonumber\\
		&&-\dfrac{\beta(t)}{2L}\|\nabla f(x(t))-\nabla f(x^*)\|^2-\frac{\rho\beta(t)}{2}\|Ax(t)-b\|^2\nonumber\\
		&&-\beta(t)\epsilon(t)\langle x(t)-x^*, x(t)\rangle,
	\end{eqnarray}    
	where the last inequality uses \eqref{SS-C}, $\dot{\epsilon}(t)\leq0$, and $\mathcal{L}_{\rho}(x(t),\lambda^*)-\mathcal{L}_{\rho}(x^*,\lambda^*)\geq0$.
	Further, we have
	\begin{eqnarray*}
		\frac{\epsilon(t)}{2}\|x(t)\|^2-\epsilon(t)\langle x(t)-x^*, x(t)\rangle=\frac{\epsilon(t)}{2}(\|x^*\|^2-\|x(t)-x^*\|^2)
		\leq\frac{\epsilon(t)}{2}\|x^*\|^2. 
	\end{eqnarray*}
	This, combined with  $\theta>\frac{1}{\alpha-1}$, $\alpha>3 $ and \eqref{zs18}, implies
	\begin{eqnarray*}
		\ddot{h}(t)+\frac{\alpha}{t}\dot{h}(t)+\theta t\dot{W}(t)&\leq&\frac{\beta(t)\epsilon(t)}{2}\|x^*\|^2
		-\dfrac{\beta(t)}{2L}\|\nabla f(x(t))-\nabla f(x^*)\|^2\nonumber\\
		&&-\frac{\rho\beta(t)}{2}\|Ax(t)-b\|^2.
	\end{eqnarray*}
\end{proof}

\begin{lemma}\label{lemma4.1.2}
Suppose that Assumption \ref{as1}  and conditions \eqref{SS-C} and \eqref{FastSR} hold.
Let $(x,\lambda): [t_{0},+\infty)\rightarrow \mathcal{X}\times\mathcal{Y}$ be a solution of System \eqref{z2} and  $(x^*,\lambda^*)\in\Omega$. Then, the positive part $[\dot{h}]_{+}$ of $\dot{h}$ satisfies $\int_{t_{0}}^{+\infty}[\dot{h}]_{+}dt<+\infty$ and the limit $\lim_{t\rightarrow+\infty}h(t)\in\mathbb{R}$ exists.
\end{lemma}
\begin{proof}
By  Lemma \ref{lemma4.1.1}, 
\begin{eqnarray*}
	\ddot{h}(t)+\frac{\alpha}{t}\dot{h}(t)+\theta t\dot{W}(t)&\leq&\frac{\beta(t)\epsilon(t)}{2}\|x^*\|^2, \quad \forall t\geq t_0.
\end{eqnarray*}
Multiplying its both sides by $t$ and then adding $\theta(\alpha+1)tW(t)$ to the both sides of the inequality, we have
\begin{eqnarray*}
 	t\ddot{h}(t)+\alpha\dot{h}(t)+\theta (t^2\dot{W}(t)+(\alpha+1)tW(t))
 	&\leq&\theta(\alpha+1)tW(t)+\frac{t\beta(t)\epsilon(t)}{2}\|x^*\|^2, \quad \forall t\geq t_0.
\end{eqnarray*}   
Multiplying its both sides by $t^{\alpha-1}$,  we have for any $t\geq t_0$,
\begin{eqnarray*}
	t^{\alpha}\ddot{h}(t)+\alpha t^{\alpha-1}\dot{h}(t)+\theta (t^{\alpha+1}\dot{W}(t)+(\alpha+1)t^\alpha W(t))
	&\leq&\theta(\alpha+1)t^{\alpha}W(t)+\frac{t^{\alpha}\beta(t)\epsilon(t)}{2}\|x^*\|^2.
\end{eqnarray*}      
This yields
\begin{eqnarray*}
	\frac{d}{dt}\left(t^{\alpha}\dot{h}(t)\right)+\theta\frac{d}{dt}\left(t^{\alpha+1}W(t)\right)\leq\theta(\alpha+1)t^{\alpha}W(t)+\frac{t^{\alpha}\beta(t)\epsilon(t)}{2}\|x^*\|^2, \quad \forall t\geq t_0.
\end{eqnarray*}   
Integrating it over  $[t_{0},t]$, we get
\begin{eqnarray*}
	t^{\alpha}\dot{h}(t)
	&\leq& t_{0}^{\alpha}\dot{h}(t_{0})+\theta t_{0}^{\alpha+1}W(t_{0})-\theta t^{\alpha+1}W(t)+\theta(\alpha+1)\int_{t_{0}}^{t}s^{\alpha}W(s)ds\\
	&&+\frac{\|x^*\|^2}{2}\int_{t_{0}}^{t}s^{\alpha}\beta(s)\epsilon(s)ds\\
	&\leq&C_3+\theta(\alpha+1)\int_{t_{0}}^{t}s^{\alpha}W(s)ds+\frac{\|x^*\|^2}{2}\int_{t_{0}}^{t}s^{\alpha}\beta(s)\epsilon(s)ds,
\end{eqnarray*} 
where $C_3\geq t_{0}^{\alpha}\dot{h}(t_{0})+\theta t_{0}^{\alpha+1}W(t_{0})$ is a nonnegative constant.
It follows that
\begin{eqnarray*}
\dot{h}(t)
&\leq&\frac{C_3}{t^{\alpha}}+\frac{\theta(\alpha+1)}{t^{\alpha}}\int_{t_{0}}^{t}s^{\alpha}W(s)ds+\frac{\|x^*\|^2}{2t^{\alpha}}\int_{t_{0}}^{t}s^{\alpha}\beta(s)\epsilon(s)ds,
\end{eqnarray*} 
and so
\begin{eqnarray*}
	[\dot{h}(t)]_{+}
	&\leq&\frac{C_3}{t^{\alpha}}+\frac{\theta(\alpha+1)}{t^{\alpha}}\int_{t_{0}}^{t}s^{\alpha}W(s)ds+\frac{\|x^*\|^2}{2t^{\alpha}}\int_{t_{0}}^{t}s^{\alpha}\beta(s)\epsilon(s)ds, \quad \forall t\geq t_0.
\end{eqnarray*} 
Integrating it over $[t_0,+\infty)$, we obtain 
\begin{eqnarray*}
	\int_{t_{0}}^{+\infty}[\dot{h}(t)]_{+}dt
	&\leq&C_3\int_{t_{0}}^{+\infty}\frac{1}{t^{\alpha}}dt+\theta(\alpha+1)\int_{t_{0}}^{+\infty}\frac{1}{t^{\alpha}}\left(\int_{t_{0}}^{t}s^{\alpha}W(s)ds\right)dt\\
	&&+\frac{\|x^*\|^2}{2}\int_{t_{0}}^{+\infty}\frac{1}{t^{\alpha}}\left(\int_{t_{0}}^{t}s^{\alpha}\beta(s)\epsilon(s)ds\right)dt\\
	&=&\frac{C_3}{(\alpha-1)t_{0}^{\alpha-1}}+\theta(\alpha+1)\int_{t_{0}}^{+\infty}\frac{1}{t^{\alpha}}\left(\int_{t_{0}}^{t}s^{\alpha}W(s)ds\right)dt\\	&&+\frac{\|x^*\|^2}{2}\int_{t_{0}}^{+\infty}\frac{1}{t^{\alpha}}\left(\int_{t_{0}}^{t}s^{\alpha}\beta(s)\epsilon(s)ds\right)dt.
\end{eqnarray*} 
Applying Lemma \ref{lemma2.1.3} with $K(t):=tW(t)$ and  $K(t):=t\beta(t)\epsilon(t)$, we get
$$\int_{t_{0}}^{+\infty}\frac{1}{t^{\alpha}}\left(\int_{t_{0}}^{t}s^{\alpha}W(s)ds\right)dt=\frac{1}{\alpha-1}\int_{t_{0}}^{+\infty}tW(t)dt$$ 
and
$$\int_{t_{0}}^{+\infty}\frac{1}{t^{\alpha}}\left(\int_{t_{0}}^{t}s^{\alpha}\beta(s)\epsilon(s)ds\right)dt=\frac{1}{\alpha-1}\int_{t_{0}}^{+\infty}t\beta(t)\epsilon(t)dt.$$
As a result, 
\begin{eqnarray}\label{sa-f3}
	\int_{t_{0}}^{+\infty}[\dot{h}(t)]_{+}dt
	\leq\frac{C_3}{(\alpha-1)t_{0}^{\alpha-1}}+\frac{\theta(\alpha+1)}{\alpha-1}\int_{t_{0}}^{+\infty}tW(t)dt	+\frac{\|x^*\|^2}{2(\alpha-1)}\int_{t_{0}}^{+\infty}t\beta(t)\epsilon(t)dt.
\end{eqnarray} 
By Theorem \ref{ztt3.1} and  using \eqref{zs17}, we have
\begin{eqnarray}\label{twint}
\int_{t_{0}}^{+\infty}tW(t)dt&=&\int_{t_{0}}^{+\infty}t\beta(t)(L_{\rho}(x(t),\lambda^*)-L_{\rho}(x^*,\lambda^*))dt+\frac{1}{2}\int_{t_{0}}^{+\infty}t\beta(t)\epsilon(t)\|x(t)\|^2dt\nonumber\\
&&+\frac{1}{2}\int_{t_{0}}^{+\infty}t(\|\dot{x}(t)\|^2+\|\dot{\lambda}(t)\|^2)dt<+\infty,
\end{eqnarray} 
which together with \eqref{sa-f3} and \eqref{FastSR} implies 
\begin{eqnarray*}
	\int_{t_{0}}^{+\infty}[\dot{h}(t)]_{+}dt<+\infty.
\end{eqnarray*} 
Define $\varphi: [t_0, +\infty)\rightarrow\mathbb{R}$ by 
$$\varphi(t)=h(t)-\int_{t_{0}}^{t}[\dot{h}(s)]_{+}ds.$$
Obviously, $\varphi(t)$ is  nonincreasing and bounded from below and so the limit $\lim_{t\rightarrow+\infty}\varphi(t)\in\mathbb{R}$ exists. As a result,  
$$\lim_{t\rightarrow+\infty}h(t)=\lim_{t\rightarrow+\infty}\varphi(t)+\int_{t_{0}}^{+\infty}[\dot{h}(t)]_{+}dt\in\mathbb{R}$$ 
exists.
\end{proof}

\begin{lemma}\label{lemma4.1.3}
	Suppose that Assumption \ref{as1}  and \eqref{SS-C} hold. 	Let $(x,\lambda): [t_{0},+\infty)\rightarrow \mathcal{X}\times\mathcal{Y}$ be a solution of System \eqref{z2} and  $(x^*,\lambda^*)\in\Omega$. Then,  for any $t\geq t_{0}$, the following inequality holds:
	\begin{eqnarray*}
		&&\quad\frac{\alpha}{t\beta(t)}\frac{d}{dt}\left(\|(\dot{x}(t), \dot{\lambda}(t))\|^2\right)+\theta\beta(t)\frac{d}{dt}\left(t\|A^T(\lambda(t)-\lambda^*)\|^2\right)\\
		&&\quad+2\langle \ddot{x}(t)+\frac{\alpha}{t}\dot{x}(t), A^T(\lambda(t)-\lambda^*)\rangle+(1-\theta)\beta(t)\|A^T(\lambda(t)-\lambda^*)\|^2\\
		&&\leq4\beta(t)\|\nabla f(x(t))-\nabla f(x^*)\|^2+4\beta(t)(\epsilon(t))^2\|x(t)\|^2+(1+2\rho^2\|A\|^2)\beta(t)\|Ax(t)-b\|^2.
	\end{eqnarray*} 
\end{lemma}
\begin{proof}
By  the definition of $\mathcal{L}_{\rho}$ and using \eqref{zcs1}, we have
\begin{eqnarray*}
	&&\left\|\beta(t)\left(\nabla f(x(t))-\nabla f(x^*)+\rho A^T(Ax(t)-b)+\epsilon(t)x(t)\right)\right\|^2\\
		&&\quad=\left\|\beta(t)\left(\nabla_{x}\mathcal{L}_{\rho}(x(t),\lambda(t)+\theta t\dot{\lambda}(t))+\epsilon(t)x(t)\right)+\beta(t)A^T(\lambda^*-\lambda(t)-\theta t\dot{\lambda}(t))\right\|^2\\
	&&\quad=\left\|\ddot{x}(t)+\frac{\alpha}{t}\dot{x}(t)+\beta(t)A^T(\lambda(t)-\lambda^*+\theta t\dot{\lambda}(t))\right\|^2\\
	&&\quad=\left\|\ddot{x}(t)+\frac{\alpha}{t}\dot{x}(t)\right\|^2+
	(\beta(t))^2\left\|A^T(\lambda(t)-\lambda^*+\theta t\dot{\lambda}(t))\right\|^2\\
	&&\qquad+2\beta(t)\langle\ddot{x}(t)+\frac{\alpha}{t}\dot{x}(t), A^T(\lambda(t)-\lambda^*)\rangle+2\theta t\beta(t)\langle\ddot{x}(t), A^T\dot{\lambda}(t)\rangle\\
	&&\qquad+2\theta\alpha\beta(t)\langle\dot{x}(t), A^T\dot{\lambda}(t)\rangle,
\end{eqnarray*}
where the second equality follows from the first equation of \eqref{z2}.
Using the second equation of  \eqref{z2}, we have
\begin{eqnarray*}
\left\|\beta(t)(Ax(t)-b)\right\|^2&=&\left\|\beta(t)\left(A(x(t)+\theta t\dot{x}(t))-b\right)-\theta t\beta(t)A\dot{x}(t)\right\|^2\\
&=&\left\|\ddot{\lambda}(t)+\frac{\alpha}{t}\dot{\lambda}(t)-\theta t\beta(t)A\dot{x}(t)\right\|^2\\
&=&\left\|\ddot{\lambda}(t)+\frac{\alpha}{t}\dot{\lambda}(t)\right\|^2+\theta^2t^2(\beta(t))^2\left\|A\dot{x}(t)\right\|^2\\
&&-2\theta t\beta(t)\langle\ddot{\lambda}(t), A\dot{x}(t)\rangle-2\theta \alpha\beta(t) \langle\dot{\lambda}(t), A\dot{x}(t)\rangle.
\end{eqnarray*}
It follows that
\begin{eqnarray*}
&&\left\|\beta(t)\left(\nabla f(x(t))-\nabla f(x^*)+\rho A^T(Ax(t)-b)+\epsilon(t)x(t)\right)\right\|^2+\left\|\beta(t)(Ax(t)-b)\right\|^2\\
&&\quad=\left\|(\ddot{x}(t),\ddot{\lambda}(t))+\frac{\alpha}{t}(\dot{x}(t),\dot{\lambda}(t))\right\|^2+
(\beta(t))^2\left\|A^T(\lambda(t)-\lambda^*+\theta t\dot{\lambda}(t))\right\|^2\\
&&\qquad+2\beta(t)\langle\ddot{x}(t)+\frac{\alpha}{t}\dot{x}(t), A^T(\lambda(t)-\lambda^*)\rangle+2\theta t\beta(t)\langle\ddot{x}(t), A^T\dot{\lambda}(t)\rangle\\
&&\qquad+\theta^2t^2(\beta(t))^2\left\|A\dot{x}(t)\right\|^2-2\theta t\beta(t)\langle\ddot{\lambda}(t), A\dot{x}(t)\rangle.
\end{eqnarray*}
Since 
\begin{eqnarray*}
&&\theta^2t^2(\beta(t))^2\left\|A\dot{x}(t)\right\|^2+2\theta t\beta(t)\langle\ddot{x}(t), A^T\dot{\lambda}(t)\rangle-2\theta t\beta(t)\langle\ddot{\lambda}(t), A\dot{x}(t)\rangle\\
&&\quad=\theta^2t^2(\beta(t))^2\left\|(A^T\dot{\lambda}(t), -A\dot{x}(t))\right\|^2-\theta^2t^2(\beta(t))^2\left\|A^T\dot{\lambda}(t)\right\|^2\\
&&\qquad+2\theta t\beta(t)\langle(\ddot{x}(t),\ddot{\lambda}(t)), (A^T\dot{\lambda}(t), -A\dot{x}(t))\rangle\\
&&\quad=\left\|\theta t\beta(t)(A^T\dot{\lambda}(t), -A\dot{x}(t))+(\ddot{x}(t),\ddot{\lambda}(t))\right\|^2-\theta^2t^2(\beta(t))^2\left\|A^T\dot{\lambda}(t)\right\|^2\\
&&\qquad-\left\|(\ddot{x}(t),\ddot{\lambda}(t))\right\|^2\\
&&\quad\geq-\theta^2t^2(\beta(t))^2\left\|A^T\dot{\lambda}(t)\right\|^2-\left\|(\ddot{x}(t),\ddot{\lambda}(t))\right\|^2
\end{eqnarray*}
and
\begin{eqnarray*}
	&&(\beta(t))^2\left\|A^T(\lambda(t)-\lambda^*+\theta t\dot{\lambda}(t))\right\|^2-\theta^2t^2(\beta(t))^2\left\|A^T\dot{\lambda}(t)\right\|^2\\
	&&\quad=(\beta(t))^2\left(\|A^T(\lambda(t)-\lambda^*)\|^2+2\theta t\langle A^T(\lambda(t)-\lambda^*), A^T\dot{\lambda}(t)\rangle\right)\\
	&&\quad=(\beta(t))^2\left(\theta\|A^T(\lambda(t)-\lambda^*)\|^2+2\theta t\langle A^T(\lambda(t)-\lambda^*), A^T\dot{\lambda}(t)\rangle\right)\\
	&&\qquad+(1-\theta)(\beta(t))^2\|A^T(\lambda(t)-\lambda^*)\|^2\\
	&&\quad=\theta(\beta(t))^2\frac{d}{dt}\left(t\|A^T(\lambda(t)-\lambda^*)\|^2\right)+(1-\theta)(\beta(t))^2\|A^T(\lambda(t)-\lambda^*)\|^2,
\end{eqnarray*}
we obtain 
\begin{eqnarray*}
	&&\left\|\beta(t)\left(\nabla f(x(t))-\nabla f(x^*)+\rho A^T(Ax(t)-b)+\epsilon(t)x(t)\right)\right\|^2+\left\|\beta(t)(Ax(t)-b)\right\|^2\\
	&&\quad\geq\left\|(\ddot{x}(t),\ddot{\lambda}(t))+\frac{\alpha}{t}(\dot{x}(t),\dot{\lambda}(t))\right\|^2-\left\|(\ddot{x}(t),\ddot{\lambda}(t))\right\|^2+(\beta(t))^2\left\|A^T(\lambda(t)-\lambda^*+\theta t\dot{\lambda}(t))\right\|^2\\
	&&\qquad-\theta^2t^2(\beta(t))^2\left\|A^T\dot{\lambda}(t)\right\|^2
	+2\beta(t)\langle\ddot{x}(t)+\frac{\alpha}{t}\dot{x}(t), A^T(\lambda(t)-\lambda^*)\rangle\\
	&&\quad=\left\|(\ddot{x}(t),\ddot{\lambda}(t))+\frac{\alpha}{t}(\dot{x}(t),\dot{\lambda}(t))\right\|^2-\left\|(\ddot{x}(t),\ddot{\lambda}(t))\right\|^2+	\theta(\beta(t))^2\frac{d}{dt}\left(t\|A^T(\lambda(t)-\lambda^*)\|^2\right)\\
	&&\qquad
    +(1-\theta)(\beta(t))^2\|A^T(\lambda(t)-\lambda^*)\|^2+2\beta(t)\langle\ddot{x}(t)+\frac{\alpha}{t}\dot{x}(t), A^T(\lambda(t)-\lambda^*)\rangle\\
    &&\quad\geq\frac{2\alpha}{t}\langle(\dot{x}(t),\dot{\lambda}(t)), (\ddot{x}(t),\ddot{\lambda}(t))\rangle+\theta(\beta(t))^2\frac{d}{dt}\left(t\|A^T(\lambda(t)-\lambda^*)\|^2\right)\\
    &&\qquad
    +(1-\theta)(\beta(t))^2\|A^T(\lambda(t)-\lambda^*)\|^2+2\beta(t)\langle\ddot{x}(t)+\frac{\alpha}{t}\dot{x}(t), A^T(\lambda(t)-\lambda^*)\rangle,
\end{eqnarray*}
which  implies
\begin{eqnarray*}
	&&\beta(t)\left\|\nabla f(x(t))-\nabla f(x^*)+\rho A^T(Ax(t)-b)+\epsilon(t)x(t)\right\|^2+\beta(t)\left\|Ax(t)-b\right\|^2\\
	&&\quad\geq\frac{2\alpha}{t\beta(t)}\langle(\dot{x}(t),\dot{\lambda}(t)), (\ddot{x}(t),\ddot{\lambda}(t))\rangle+\theta\beta(t)\frac{d}{dt}\left(t\|A^T(\lambda(t)-\lambda^*)\|^2\right)\\
	&&\qquad
	+(1-\theta)\beta(t)\|A^T(\lambda(t)-\lambda^*)\|^2+2\langle\ddot{x}(t)+\frac{\alpha}{t}\dot{x}(t), A^T(\lambda(t)-\lambda^*)\rangle\\
	&&\quad=\frac{\alpha}{t\beta(t)}\frac{d}{dt}\left(\|(\dot{x}(t),\dot{\lambda}(t))\|^2\right)+\theta\beta(t)\frac{d}{dt}\left(t\|A^T(\lambda(t)-\lambda^*)\|^2\right)\\
	&&\qquad+(1-\theta)\beta(t)\|A^T(\lambda(t)-\lambda^*)\|^2+2\langle\ddot{x}(t)+\frac{\alpha}{t}\dot{x}(t), A^T(\lambda(t)-\lambda^*)\rangle.
\end{eqnarray*}
Since 
\begin{eqnarray*}
&&\left\|\nabla f(x(t))-\nabla f(x^*)+\rho A^T(Ax(t)-b)+\epsilon(t)x(t)\right\|^2\\
&&\quad\leq2\left\|\nabla f(x(t))-\nabla f(x^*)+\epsilon(t)x(t)\right\|^2+2\rho^2\|A\|^2\|Ax(t)-b\|^2, 
\end{eqnarray*}
we have 
\begin{eqnarray*}
	&&\quad\frac{\alpha}{t\beta(t)}\frac{d}{dt}\left(\|(\dot{x}(t), \dot{\lambda}(t))\|^2\right)+\theta\beta(t)\frac{d}{dt}\left(t\|A^T(\lambda(t)-\lambda^*)\|^2\right)\\
	&&\quad+2\langle \ddot{x}(t)+\frac{\alpha}{t}\dot{x}(t), A^T(\lambda(t)-\lambda^*)\rangle+(1-\theta)\beta(t)\|A^T(\lambda(t)-\lambda^*)\|^2\\
	&&\leq4\beta(t)\|\nabla f(x(t))-\nabla f(x^*)\|^2+4\beta(t)(\epsilon(t))^2\|x(t)\|^2+(1+2\rho^2\|A\|^2)\beta(t)\|Ax(t)-b\|^2.
\end{eqnarray*} 
\end{proof}

\begin{lemma}\label{lemma4.1.4}
	Suppose that Assumption \ref{as1} and conditions \eqref{SS-C}  and \eqref{FastSR}  hold.
	Let $(x,\lambda): [t_{0},+\infty)\rightarrow \mathcal{X}\times\mathcal{Y}$ be a solution of System \eqref{z2} and  $(x^*,\lambda^*)\in\Omega$. Then, it holds:
	$$\int_{t_{0}}^{+\infty}t\beta(t)\|A^T(\lambda(t)-\lambda^*)\|^2dt
	<+\infty.$$
\end{lemma}

\begin{proof}
By  Lemma \ref{lemma4.1.1} and Lemma \ref{lemma4.1.3},
\begin{eqnarray*}
	&&\quad\ddot{h}(t)+\frac{\alpha}{t}\dot{h}(t)+\theta t\dot{W}(t)+\frac{\alpha}{t\beta(t)}\frac{d}{dt}\left(\|(\dot{x}(t), \dot{\lambda}(t))\|^2\right)+\theta\beta(t)\frac{d}{dt}\left(t\|A^T(\lambda(t)-\lambda^*)\|^2\right)\\
	&&\quad+2\langle \ddot{x}(t)+\frac{\alpha}{t}\dot{x}(t), A^T(\lambda(t)-\lambda^*)\rangle\\
	&&\leq(4-\dfrac{1}{2L})\beta(t)\|\nabla f(x(t))-\nabla f(x^*)\|^2
	+(1+2\rho^2\|A\|^2-\frac{\rho}{2})\beta(t)\|Ax(t)-b\|^2\\
	&&\quad+\frac{\beta(t)\epsilon(t)}{2}\|x^*\|^2+4\beta(t)(\epsilon(t))^2\|x(t)\|^2+(\theta-1)\beta(t)\|A^T(\lambda(t)-\lambda^*)\|^2\\
	&&\leq C_4\beta(t)\|\nabla f(x(t))-\nabla f(x^*)\|^2+C_5\beta(t)\|Ax(t)-b\|^2\\
	&&\quad+\frac{\beta(t)\epsilon(t)}{2}\|x^*\|^2+4\beta(t)(\epsilon(t))^2\|x(t)\|^2+(\theta-1)\beta(t)\|A^T(\lambda(t)-\lambda^*)\|^2,
\end{eqnarray*}	
where $C_4=[4-\dfrac{1}{2L}]_{+}$ and $C_5=[1+2\rho^2\|A\|^2-\frac{\rho}{2}]_{+}$. Multiplying the above inequality by $t^{\alpha}$ and then integrating over $[t_{0},t]$,  we have 
\begin{eqnarray}\label{zd5}
 &&J_1(t)+\theta J_2(t)+\alpha J_3(t)+\theta J_4(t)+2J_5(t)\nonumber\\
	&&\quad\leq C_4\int_{t_{0}}^{t}s^{\alpha}\beta(s)\|\nabla f(x(s))-\nabla f(x^*)\|^2ds+C_5\int_{t_{0}}^{t}s^{\alpha}\beta(s)\|Ax(s)-b\|^2ds\nonumber\\
	&&\qquad+\frac{\|x^*\|^2}{2}\int_{t_{0}}^{t}s^{\alpha}\beta(s)\epsilon(s)ds+4\int_{t_{0}}^{t}s^{\alpha}\beta(s)(\epsilon(s))^2\|x(s)\|^2ds\nonumber\\
	&&\qquad+(\theta-1)\int_{t_{0}}^{t}s^{\alpha}\beta(s)\|A^T(\lambda(s)-\lambda^*)\|^2ds, \quad\forall t\geq t_0,
\end{eqnarray}	
where 
$$J_1(t):=\int_{t_{0}}^{t}(s^{\alpha}\ddot{h}(s)+\alpha s^{\alpha-1}\dot{h}(s))ds, \quad J_2(t):=\int_{t_{0}}^{t} s^{\alpha+1}\dot{W}(s)ds,$$
$$J_3(t):=\int_{t_{0}}^{t}\frac{ s^{\alpha-1}}{\beta(s)}d\left(\|(\dot{x}(s), \dot{\lambda}(s))\|^2\right), \quad J_4(t):=\int_{t_{0}}^{t} s^{\alpha}\beta(s)d\left(s\|A^T(\lambda(s)-\lambda^*)\|^2\right),$$
$$J_5(t):=\int_{t_{0}}^{t}\langle s^{\alpha}\ddot{x}(s)+\alpha s^{\alpha-1}\dot{x}(s), A^T(\lambda(s)-\lambda^*)\rangle ds.$$

Next we continue to compute $J_1(t), J_2(t), J_3(t), J_4(t), J_5(t)$.   Let $t\ge t_0$.
$$J_1(t)=\int_{t_{0}}^{t}(s^{\alpha}\ddot{h}(s)+\alpha s^{\alpha-1}\dot{h}(s))ds=\int_{t_{0}}^{t}d(s^{\alpha}\dot{h}(s))=t^{\alpha}\dot{h}(t)-t_0^{\alpha}\dot{h}(t_0),$$
which yields
\begin{eqnarray}\label{zd1}
0=J_1(t)-t^{\alpha}\dot{h}(t)+t_0^{\alpha}\dot{h}(t_0)\leq J_1(t) -t^{\alpha}\dot{h}(t)+t_0^{\alpha}|\dot{h}(t_0)|.
\end{eqnarray}
By using integration by parts, we get
$$J_2(t)=\int_{t_{0}}^{t}s^{\alpha+1}d(W(s))=t^{\alpha+1}W(t)-t_0^{\alpha+1}W(t_0)-(\alpha+1)\int_{t_{0}}^{t}s^{\alpha}W(s)ds,$$
which  implies 
\begin{eqnarray}\label{zd2}
0\leq t^{\alpha+1}W(t)=J_2(t)+t_0^{\alpha+1}W(t_0)+(\alpha+1)\int_{t_{0}}^{t}s^{\alpha}W(s)ds.
\end{eqnarray}
Applying  integration by parts to $J_3(t)$, we have
\begin{eqnarray*}
	J_3(t	)&=&\frac{ t^{\alpha-1}}{\beta(t)}\|(\dot{x}(t), \dot{\lambda}(t))\|^2-\frac{ t_{0}^{\alpha-1}}{\beta(t_{0})}\|(\dot{x}(t_{0}), \dot{\lambda}(t_{0}))\|^2-(\alpha-1)\int_{t_{0}}^{t}\frac{s^{\alpha-2}}{\beta(s)}\|(\dot{x}(s), \dot{\lambda}(s))\|^2ds\\
	&&+\int_{t_{0}}^{t}\frac{s^{\alpha-1}\dot{\beta}(s)}{(\beta(s))^2}\|(\dot{x}(s), \dot{\lambda}(s))\|^2ds\\
	&\geq&\frac{ t^{\alpha-1}}{\beta(t)}\|(\dot{x}(t), \dot{\lambda}(t))\|^2-\frac{ t_{0}^{\alpha-1}}{\beta(t_{0})}\|(\dot{x}(t_{0}), \dot{\lambda}(t_{0}))\|^2-\frac{\alpha-1}{t_{0}^2\beta(t_{0})}\int_{t_{0}}^{t}s^{\alpha}\|(\dot{x}(s), \dot{\lambda}(s))\|^2ds,
\end{eqnarray*}		
which yields
\begin{eqnarray}\label{zd3}
	0\leq\frac{ t^{\alpha-1}}{\beta(t)}\|(\dot{x}(t), \dot{\lambda}(t))\|^2\leq
	J_3(t)+\frac{ t_{0}^{\alpha-1}}{\beta(t_{0})}\|(\dot{x}(t_{0}), \dot{\lambda}(t_{0}))\|^2+\frac{\alpha-1}{t_{0}^2\beta(t_{0})}\int_{t_{0}}^{t}s^{\alpha}\|(\dot{x}(s), \dot{\lambda}(s))\|^2ds.
\end{eqnarray}	
Integration by parts to $J_4(t)$ gives	
\begin{eqnarray*}
	J_4(t)&=&t^{\alpha+1}\beta(t)\|A^T(\lambda(t)-\lambda^*)\|^2-t_0^{\alpha+1}\beta(t_0)\|A^T(\lambda(t_0)-\lambda^*)\|^2\\
	&&-\int_{t_{0}}^{t}s(\alpha s^{\alpha-1}\beta(s)+s^\alpha\dot{\beta}(s))\|A^T(\lambda(s)-\lambda^*)\|^2ds.
\end{eqnarray*}		
Remembering \eqref{SS-C}, from \eqref{zs19} we get
\begin{eqnarray*}
\alpha s^{\alpha-1}\beta(s)+s^\alpha\dot{\beta}(s)\le\left(\frac{1-2\theta}{\theta}-b+\alpha\right)s^{\alpha-1}\beta(s), \quad \forall s\geq t_0.
\end{eqnarray*}			
Thus, 
\begin{eqnarray}\label{zd4}
t^{\alpha+1}\beta(t)\|A^T(\lambda(t)-\lambda^*)\|^2&\le&
	J_4(t)+t_0^{\alpha+1}\beta(t_0)\|A^T(\lambda(t_0)-\lambda^*)\|^2\\
	&&+\left(\frac{1-2\theta}{\theta}-b+\alpha\right)\int_{t_{0}}^{t}s^{\alpha}\beta(s)\|A^T(\lambda(s)-\lambda^*)\|^2ds.\nonumber
\end{eqnarray}
Again by integration by parts to $J_5(t)$, we get		
\begin{eqnarray*}
	J_5(t)&=&\int_{t_{0}}^{t}\langle \frac{d}{ds}\left(s^{\alpha}\dot{x}(s)\right), A^T(\lambda(s)-\lambda^*)\rangle ds\\
	&=&t^{\alpha}\langle A^T(\lambda(t)-\lambda^*), \dot{x}(t)\rangle-t_0^{\alpha}\langle A^T(\lambda(t_0)-\lambda^*), \dot{x}(t_0)\rangle-\int_{t_{0}}^{t}s^{\alpha}\langle\dot{x}(s), A^T\dot{\lambda}(s)\rangle ds\\
&\ge&t^{\alpha}\langle A^T(\lambda(t)-\lambda^*), \dot{x}(t)\rangle-t_0^{\alpha}\langle A^T(\lambda(t_0)-\lambda^*), \dot{x}(t_0)\rangle-\frac{\max\{1, \|A\|^2\}}{2}\int_{t_{0}}^{t}s^{\alpha}(\|\dot{x}(s)\|^2+\|\dot{\lambda}(s)\|^2) ds,
\end{eqnarray*}	
which  leads to
\begin{eqnarray}\label{saj5}
0&\leq& J_5(t)-t^{\alpha}\langle A^T(\lambda(t)-\lambda^*), \dot{x}(t)\rangle+t_0^{\alpha}|\langle A^T(\lambda(t_0)-\lambda^*), \dot{x}(t_0)\rangle|\nonumber\\
&&+\frac{\max\{1, \|A\|^2\}}{2}\int_{t_{0}}^{t}s^{\alpha}(\|\dot{x}(s)\|^2+\|\dot{\lambda}(s)\|^2) ds.
\end{eqnarray}	
Combining  \eqref{zd1}, \eqref{zd2}, \eqref{zd3}, \eqref{zd4}, and \eqref{saj5} together, we obtain 
\begin{eqnarray*}
	&&\theta t^{\alpha+1}\beta(t)\|A^T(\lambda(t)-\lambda^*)\|^2\\
	&&\quad\leq J_1(t)+\theta J_2(t)+\alpha J_3(t)+\theta J_4(t)+2J_5(t)-t^{\alpha}\dot{h}(t)\\
	&&\qquad+\int_{t_{0}}^{t}s^{\alpha}\left((\alpha+1)\theta W(s)+\left(\frac{\alpha(\alpha-1)}{t_{0}^2\beta(t_{0})}+\max\{1, \|A\|^2\}\right)\|(\dot{x}(s), \dot{\lambda}(s))\|^2\right)ds\\
	&&\qquad+\left(1-2\theta+(\alpha-b)\theta\right)\int_{t_{0}}^{t}s^{\alpha}\beta(s)\|A^T(\lambda(s)-\lambda^*)\|^2ds-2t^{\alpha}\langle A^T(\lambda(t)-\lambda^*), \dot{x}(t)\rangle+C_6,
\end{eqnarray*}	
where
$C_6$ is a nonnegative constant.  This together with \eqref{zd5} yields 
\begin{eqnarray}\label{zd7}
	&&\theta t^{\alpha+1}\beta(t)\|A^T(\lambda(t)-\lambda^*)\|^2\nonumber\\
	&\le&-t^{\alpha}\dot{h}(t)+(\alpha-b-1)\theta\int_{t_{0}}^{t}s^{\alpha}\beta(s)\|A^T(\lambda(s)-\lambda^*)\|^2ds\nonumber\\	
&&\qquad-2t^{\alpha}\langle A^T(\lambda(t)-\lambda^*), \dot{x}(t)\rangle+\int_{t_{0}}^{t}s^{\alpha}G(s)ds+C_6, \quad \forall t \geq t_0,
\end{eqnarray}	
where 
\begin{eqnarray}\label{zd11}
G(s)&=&(\alpha+1)\theta W(s)+\left(\frac{\alpha(\alpha-1)}{t_0^2\beta(t_0)}+\max\{1, \|A\|^2\}\right)\|(\dot{x}(s), \dot{\lambda}(s))\|^2\nonumber\\
&&+C_4\beta(s)\|\nabla f(x(s))-\nabla f(x^*)\|^2+C_5\beta(s)\|Ax(s)-b\|^2+\frac{\|x^*\|^2}{2}\beta(s)\epsilon(s)\nonumber\\
&&+4\beta(s)(\epsilon(s))^2\|x(s)\|^2, \quad \forall s\geq t_0.
\end{eqnarray}
Dividing both sides of \eqref{zd7} by $t^{\alpha}$, we have 
\begin{eqnarray*}
\theta t\beta(t)\|A^T(\lambda(t)-\lambda^*)\|^2
	&\leq&-\dot{h}(t)+\frac{(\alpha-b-1)\theta}{t^{\alpha}}\int_{t_{0}}^{t}s^{\alpha}\beta(s)\|A^T(\lambda(s)-\lambda^*)\|^2ds\nonumber\\
	&&-2\langle A^T(\lambda(t)-\lambda^*), \dot{x}(t)\rangle+\frac{1}{t^{\alpha}}\int_{t_{0}}^{t}s^{\alpha}G(s)ds+\frac{C_6}{t^{\alpha}}, \quad \forall t \geq t_0.
\end{eqnarray*}	
Take $\tau \geq t_0$. Integrating the above inequality from $t_0$ to $\tau$ gives 
\begin{eqnarray}\label{intf}
	&&\theta \int_{t_{0}}^{\tau}t\beta(t)\|A^T(\lambda(t)-\lambda^*)\|^2dt\nonumber\\
	&&\quad\leq h(t_0)-h(\tau)
	+(\alpha-b-1)\theta\int_{t_{0}}^{\tau}\frac{1}{t^{\alpha}}\left(\int_{t_{0}}^{t}s^{\alpha}\beta(s)\|A^T(\lambda(s)-\lambda^*)\|^2ds\right)dt\nonumber\\
	&&\qquad-2\int_{t_{0}}^{\tau}\langle A^T(\lambda(t)-\lambda^*), \dot{x}(t)\rangle dt
	+\int_{t_{0}}^{\tau}\frac{1}{t^{\alpha}}\left(\int_{t_{0}}^{t}s^{\alpha}G(s)ds\right)dt+C_6\int_{t_{0}}^{\tau}\frac{1}{t^{\alpha}}dt.
\end{eqnarray}
Now we estimate the four integration terms in the right side of \eqref{intf}.
\begin{equation}\label{intf1}
\int_{t_{0}}^{\tau}\frac{1}{t^{\alpha}}dt=\frac{\tau^{1-\alpha}}{1-\alpha}+\frac{t_{0}^{1-\alpha}}{\alpha-1}\leq\frac{t_{0}^{1-\alpha}}{\alpha-1}.
\end{equation}

By Lemma \ref{lemma2.1.3} with $K(t)=tG(t)$ and  $K(t)=t\beta(t)\|A^T(\lambda(t)-\lambda^*)\|^2$, we get 
\begin{equation}\label{intf2}
\int_{t_{0}}^{\tau}\frac{1}{t^{\alpha}}\left(\int_{t_{0}}^{t}s^{\alpha}G(s)ds\right)dt\leq\frac{1}{\alpha-1}\int_{t_{0}}^{\tau}tG(t)dt
\end{equation}
and
\begin{equation}\label{intf3}
\int_{t_{0}}^{\tau}\frac{1}{t^{\alpha}}\left(\int_{t_{0}}^{t}s^{\alpha}\beta(s)\|A^T(\lambda(s)-\lambda^*)\|^2ds\right)dt\leq\frac{1}{\alpha-1}\int_{t_{0}}^{\tau}t\beta(t)\|A^T(\lambda(t)-\lambda^*)\|^2dt.
\end{equation}
Further, we have
\begin{eqnarray}\label{intf4}
-\int_{t_{0}}^{\tau}\langle A^T(\lambda(t)-\lambda^*), \dot{x}(t)\rangle dt
&=&\langle \lambda(t_0)-\lambda^*, Ax(t_0)-b\rangle-\langle \lambda(\tau)-\lambda^*, Ax(\tau)-b\rangle\nonumber\\
&&+\int_{t_{0}}^{\tau}\langle Ax(t)-b, \dot{\lambda}(t)\rangle dt\nonumber\\
&\leq&\|\lambda(t_0)-\lambda^*\|\|Ax(t_0)-b\|+\|\lambda(\tau)-\lambda^*\|\|Ax(\tau)-b\|\nonumber\\
&&+\frac{1}{2}\int_{t_{0}}^{\tau}t\beta(t)\|Ax(t)-b\|^2dt+\frac{1}{2}\int_{t_{0}}^{\tau}\frac{\|\dot{\lambda}(t)\|^2}{t\beta(t)}dt\nonumber\\
&\leq&\frac{1}{2}\int_{t_{0}}^{\tau}\frac{t\|\dot{\lambda}(t)\|^2}{t^2\beta(t)}dt+C_7,
\end{eqnarray}	
where $C_7>0$ is a constant and the last inequality uses the  facts that $\|Ax(t)-b\|=\mathcal{O}\left(\frac{1}{t^2\beta(t)}\right)$ and  $(\lambda(t))_{t\geq t_{0}}$ is bounded, and
$\int_{t_{0}}^{+\infty}t\beta(t)\|Ax(t)-b\|^2dt<+\infty$ which are guaranteed by Theorem \ref{ztt3.1}.
Remember that the constant $b>0$ appearing in \eqref{zs19} can be taken small enough such $\alpha-b-1>0$. It follows from \eqref{intf},\eqref{intf1},\eqref{intf2},\eqref{intf3}, and \eqref{intf4} that 
\begin{eqnarray*}
	&&\theta \int_{t_{0}}^{\tau}t\beta(t)\|A^T(\lambda(t)-\lambda^*)\|^2dt\\
	&&\quad\leq h(t_0)-h(\tau)
	+\frac{(\alpha-b-1)\theta}{\alpha-1}\int_{t_{0}}^{\tau}t\beta(t)\|A^T(\lambda(t)-\lambda^*)\|^2dt\\
	&&\qquad+\int_{t_{0}}^{\tau}\frac{t\|\dot{\lambda}(t)\|^2}{t^2\beta(t)}dt
	+\frac{1}{\alpha-1}\int_{t_{0}}^{\tau}tG(t)dt+\frac{C_6t_{0}^{1-\alpha}}{\alpha-1}+2C_7\\
&&\quad\leq \frac{(\alpha-b-1)\theta}{\alpha-1}\int_{t_{0}}^{\tau}t\beta(t)\|A^T(\lambda(t)-\lambda^*)\|^2dt\\
	&&\qquad	+\frac{1}{\alpha-1}\int_{t_{0}}^{\tau}tG(t)dt+C_8,
\end{eqnarray*}
where $C_8$ is a positive constant and the last inequality uses  the fact $h(t)\ge 0$ and 
$$\int_{t_{0}}^{\tau}\frac{t\|\dot{\lambda}(t)\|^2}{t^2\beta(t)}dt\le \frac{1}{t_{0}^2\beta(t_{0})}\int_{t_{0}}^{\tau}t\|\dot{\lambda}(t)\|^2dt<+\infty,$$
which is guaranteed by  Theorem \ref{ztt3.1}. As a result,
\begin{eqnarray*}
	\frac{b\theta}{\alpha-1} \int_{t_{0}}^{+\infty}t\beta(t)\|A^T(\lambda(t)-\lambda^*)\|^2dt\leq\frac{1}{\alpha-1}\int_{t_{0}}^{+\infty}tG(t)dt+C_8.
\end{eqnarray*}	
To complete the proof, we need to prove 
$$\int_{t_{0}}^{+\infty}tG(t)dt<+\infty$$
which follows from  Theorem \ref{ztt3.1}, \eqref{zd11}, \eqref{twint}, \eqref{FastSR}, and $\lim_{t\to+\infty}\epsilon(t)=0$. 
\end{proof}

Next, we derive some  convergence rate results for KKT system associated with Problem \eqref{z1}, which will be used to prove that the weak convergence of the trajectory of System \eqref{z2}.
\begin{theorem}\label{theorem4.1.5}
	Suppose that Assumption \ref{as1} and conditions \eqref{SS-C} and \eqref{FastSR}  hold.
	Let $(x,\lambda): [t_{0},+\infty)\rightarrow \mathcal{X}\times\mathcal{Y}$ be a solution of System \eqref{z2} and  $(x^*,\lambda^*)\in\Omega$. Then, it holds:
	$$\|\nabla_{x}\mathcal{L}(x(t),\lambda(t))\|=\|\nabla f(x(t))+A^T(\lambda(t)\|=o\left(\frac{1}{\sqrt{t}(\beta(t))^{\frac{1}{4}}}\right)$$ 
and
$$ \|\nabla_{\lambda}\mathcal{L}(x(t),\lambda(t))\|=\|Ax(t)-b\|=\mathcal{O}\left(\frac{1}{t^2\beta(t)}\right).$$
\end{theorem}

\begin{proof}
Since  $\beta(t)>0$ for all $t\ge t_0$ is nondedreasing, we have $(\beta(t))^\frac{1}{2}\le t\beta(t)$ for all $t$ large enough.
It follows from \eqref{SS-C} and \eqref{zs19} that  
\begin{eqnarray*}
(\beta(t))^{\frac{1}{2}}+\frac{t}{2}(\beta(t))^{-\frac{1}{2}}\dot{\beta}(t)\leq\frac{1-b\theta}{2\theta}(\beta(t))^{\frac{1}{2}}\le \frac{1-b\theta}{2\theta}t\beta(t)
\end{eqnarray*}
for all $t$ large enough,  where $b>0$ is a small enough constant.

It follows that for all $t$ large enough, 
\begin{eqnarray}\label{sun1}
&&\frac{d}{dt}\left(t(\beta(t))^{\frac{1}{2}}\|A^T(\lambda(t)-\lambda^*)\|^2\right)\nonumber\\
&&\quad=\left((\beta(t))^{\frac{1}{2}}+\frac{t}{2}(\beta(t))^{-\frac{1}{2}}\dot{\beta}(t)\right)\|A^T(\lambda(t)-\lambda^*)\|^2+2t(\beta(t))^{\frac{1}{2}}\langle AA^T(\lambda(t)-\lambda^*), \dot{\lambda}(t)\rangle\nonumber\\
&&\quad\leq\frac{1-b\theta}{2\theta}t\beta(t)\|A^T(\lambda(t)-\lambda^*)\|^2+t\beta(t)\|A\|^2\cdot\|A^T(\lambda(t)-\lambda^*)\|^2+t\|\dot{\lambda}(t)\|^2
\end{eqnarray}
and 
\begin{eqnarray}\label{sun2}
	&&\frac{d}{dt}\left(t(\beta(t))^{\frac{1}{2}}\|\nabla f(x(t))-\nabla f(x^*)\|^2\right)\nonumber\\
	&&\quad=\left((\beta(t))^{\frac{1}{2}}+\frac{t}{2}(\beta(t))^{-\frac{1}{2}}\dot{\beta}(t)\right)\|\nabla f(x(t))-\nabla f(x^*)\|^2\nonumber\\
	&&\qquad+2t(\beta(t))^{\frac{1}{2}}\langle\nabla f(x(t))-\nabla f(x^*), \frac{d}{dt}\nabla f(x(t))\rangle\nonumber\\
	&&\quad \leq\frac{1-b\theta}{2\theta}t\beta(t)\|\nabla f(x(t))-\nabla f(x^*)\|^2+t\beta(t)\|\nabla f(x(t))-\nabla f(x^*)\|^2\nonumber\\
	&&\qquad+t\|\frac{d}{dt}\nabla f(x(t))\|^2\nonumber\\
	&&\quad \leq\frac{1+\theta(2-b)}{2\theta}t\beta(t)\|\nabla f(x(t))-\nabla f(x^*)\|^2
	+L^2t\|\dot{x}(t)\|^2,
\end{eqnarray}
where the last inequality uses $\|\frac{d}{dt}\nabla f(x(t))\|\le L\|\dot{x}(t)\|$ which is guaranteed by $L$-Lipschitz continuity of  $\nabla f$.
By Theorem \ref{ztt3.1} and Lemma \ref{lemma4.1.4}, we obtain,
\begin{eqnarray}\label{ad14}
	\int_{t_0}^{+\infty}\left(\frac{1-b\theta}{2\theta}+\|A\|^2\right)t\beta(t)\|A^T(\lambda(t)-\lambda^*)\|^2+t\|\dot{\lambda}(t)\|^2dt<+\infty
\end{eqnarray}
and 
\begin{eqnarray}\label{ad15}
	\int_{t_0}^{+\infty}\frac{1+\theta(2-b)}{2\theta}t\beta(t)\|\nabla f(x(t))-\nabla f(x^*)\|^2dt<+\infty.
\end{eqnarray}
By Lemma \ref{lemma2.1.5}, \eqref{sun1}, \eqref{sun2}, \eqref{ad14}, and \eqref{ad15},  we have
$$\lim_{t\rightarrow+\infty}t(\beta(t))^{\frac{1}{2}}\|A^T(\lambda(t)-\lambda^*)\|^2=0$$
and 
$$\lim_{t\rightarrow+\infty}t(\beta(t))^{\frac{1}{2}}\|\nabla f(x(t))-\nabla f(x^*)\|^2=0,$$
which means 
$$\|A^T(\lambda(t)-\lambda^*)\|=o\left(\frac{1}{\sqrt{t}(\beta(t))^{\frac{1}{4}}}\right)$$
 and
$$\|\nabla f(x(t))-\nabla f(x^*)\|=o\left(\frac{1}{\sqrt{t}(\beta(t))^{\frac{1}{4}}}\right).$$
Consequently,
$$\|\nabla_{x}\mathcal{L}(x(t),\lambda(t))\|=o\left(\frac{1}{\sqrt{t}(\beta(t))^{\frac{1}{4}}}\right)$$
since
\begin{eqnarray*}
	\|\nabla_{x}\mathcal{L}(x(t),\lambda(t))\|&=&\|\nabla f(x(t))+A^T\lambda(t)\|\nonumber\\
		&\leq&\|\nabla f(x(t))-\nabla f(x^*)\|+\|A^T(\lambda(t)-\lambda^*)\|.
\end{eqnarray*}
The result 
$$ \|\nabla_{\lambda}\mathcal{L}(x(t),\lambda(t))\|=\|Ax(t)-b\|=\mathcal{O}\left(\frac{1}{t^2\beta(t)}\right)$$
follows directly from  (ii) of  Theorem \ref{ztt3.1}.
\end{proof}

Now we present  the main result of this section.

\begin{theorem}\label{theorem4.1.6}
	Suppose that Assumption \ref{as1} and conditions \ref{SS-C} and \eqref{FastSR} hold.	Let $(x,\lambda): [t_{0},+\infty)\rightarrow \mathcal{X}\times\mathcal{Y}$ be a solution of System \eqref{z2} and  $(x^*,\lambda^*)\in\Omega$. Then, the trajectory $(x(t), \lambda(t))$ of System \eqref{z2} converges weakly to a primal\mbox{-}dual optimal solution of  Problem \eqref{z1} as $t\to +\infty$.
\end{theorem}
 
\begin{proof}
By Lemma \ref{lemma4.1.2},  the limit $\lim_{t\rightarrow+\infty}\|(x(t), \lambda(t))-(x^*, \lambda^*)\|\in \mathbb{R}$ exists  for any $(x^*, \lambda^*)\in \Omega$. This proves condition (i) of Opial's Lemma \ref{lemma2.1.6}. 

Let $(\widetilde{x}, \widetilde{\lambda})\in \mathcal{X}\times\mathcal{Y}$ be an 
arbitrary weak sequential cluster point of $(x(t), \lambda(t))_{t\geq t_0}$. Then, there exists  $t_n\ge t_0$ with $\lim_{n\to+\infty}t_n=+\infty$ such that 
\begin{eqnarray*}
	(x(t_n), \lambda(t_n))\rightharpoonup(\widetilde{x}, \widetilde{\lambda})   \quad as \quad n\rightarrow +\infty.
\end{eqnarray*}
By  Theorem \ref{theorem4.1.5}, we obtain
$$\lim_{t\rightarrow+\infty}\|\nabla_{x}\mathcal{L}(x(t_n),\lambda(t_n))\|=0\quad\text{and} \quad \lim_{t\rightarrow+\infty}\|\nabla_{\lambda}\mathcal{L}(x(t_n),\lambda(t_n))\|=0.$$
Define $\mathcal{T}_{\mathcal{L}}: \mathcal{X}\times\mathcal{Y}\rightarrow \mathcal{X}\times\mathcal{Y}$ by 
\begin{eqnarray*}
	\mathcal{T}_{\mathcal{L}}(x, \lambda)=\begin{pmatrix}\nabla_{x}\mathcal{L}(x, \lambda) \\-\nabla_{\lambda}\mathcal{L}(x, \lambda)\end{pmatrix}.
\end{eqnarray*}
The operator $\mathcal{T}_{\mathcal{L}}$ is maximally monotone and so its graph is sequentially closed in $(\mathcal{X}\times\mathcal{Y})^{\text{weak}}\times(\mathcal{X}\times\mathcal{Y})^{\text{strong}}$ (see \cite[Proposition $20.38$]{BauschkeCombet2017}). Thus, $(0,0)=\mathcal{T}_{\mathcal{L}}(\widetilde{x}, \widetilde{\lambda})$, which means
\begin{eqnarray*}
	\begin{cases}
		\nabla f(\widetilde{x})+A^T\widetilde{\lambda}=0,\\
		A\widetilde{x}-b=0,
	\end{cases}
\end{eqnarray*}
which together with \eqref{zcs1} implies $(\widetilde{x}, \widetilde{\lambda})\in\Omega$. This proves condition (ii) of Opial's Lemma.

By the Opial's Lemma,  $(x(t), \lambda(t))$ converges weakly to  a primal\mbox{-}dual optimal solutions of Problem \eqref{z1} as $t\to +\infty$.
\end{proof}

\begin{remark}\label{rmkweak}
Bot and Nguyen \cite[Theorem 4.8]{BNguyen2022} showed that the  trajectory generated by $\text{(Z-AVD)}$ with $\alpha>3$ converges weakly to a primal-dual optimal solution of Problem \eqref{z1}.  When Assumption \ref{as1} and the strict scaling condition \eqref{SS-C} hold, Hulett and Nguyen \cite[Theorem 4.10]{HulettNeuyen2023} further establised the weak convergence of the trajectory   generated by  $\text{(HN-AVD)}$ to a primal-dual optimal solution. As shown in Theorem \ref{theorem4.1.5}, the property on the weak convergence of the trajectory to a primal-dual solution retains true  forS ystem \eqref{z2}  when the rescaled regularization parameter satisfies the condition \eqref{FastSR}.
\end{remark}

\section{Strong convergence of the primal  trajectory to the minimal norm solution}

In this section, following the approaches presented in Attouch et al. \cite{AttouchZH2018} and Zhu et al. \cite{zhuhufang1},  we will discuss the strong convergence of the primal trajectory $(x(t))_{t\geq t_{0}}$ of System \eqref{z2} to the minimal norm  solution of Problem \eqref{z1} under the condition
\begin{equation}\label{Slow}
\lim_{t\rightarrow+\infty}t^2\beta(t)\epsilon(t)=+\infty,
\end{equation}
which reflects the case  that the rescaled regularization parameter $\beta(t)\epsilon(t)$ vanishes slowly.

Let $\bar{x}^*\in\mathcal{X}$ be the minimal norm solution of Problem \eqref{z1}, i.e., $\bar{x}^*=\text{Proj}_{S}0$, where  $S$ is the solution set of Problem \eqref{z1} and $\text{Proj}_S$ is the projection operator associated with $S$.  Then, there exists a solution $\bar{\lambda}^*\in\mathcal{Y}$ of the dual problem \eqref{dudu5} such that $(\bar{x}^*,\bar{\lambda}^*)\in\Omega$. For every $\epsilon>0$, define $\mathcal{L}_{\epsilon} : \mathcal{X}\rightarrow\mathbb{R}$  by
\begin{eqnarray}\label{zf1}
	\mathcal{L}_{\epsilon}(x):=\mathcal{L}_{\rho}(x,\bar{\lambda}^*)+\frac{\epsilon}{2}\|x\|^2,\quad \forall x\in  \mathcal{X}.
\end{eqnarray}
Since $\mathcal{L}_{\epsilon}$ is strongly convex,  and it  has a unique minimizer $x_{\epsilon}$, i.e.,
$$x_{\epsilon}:=\arg\min_{x\in\mathcal{X}}\mathcal{L}_{\epsilon}(x).$$
It is well-known that
\begin{eqnarray}\label{zf3}
	\lim_{\epsilon\rightarrow0}\|x_{\epsilon}-\bar{x}^*\|=0,\quad \|x_{\epsilon}\|\leq\|\bar{x}^*\|, \quad \forall \epsilon> 0,
\end{eqnarray}
and
\begin{eqnarray}\label{zf2}
	\nabla\mathcal{L}_{\epsilon}(x_{\epsilon})=\nabla_{x}\mathcal{L}_{\rho}(x_{\epsilon},\bar{\lambda}^*)+\epsilon x_{\epsilon}=0.
\end{eqnarray}
From \eqref{zcs1}, we have $\bar{x}^*\in \mathcal{D}(\bar{\lambda}^*)$, where $\mathcal{D}(\bar{\lambda}^*)=\arg\min_{x\in\mathcal{X}}\mathcal{L}_{\rho}(x, \bar{\lambda}^*)$. From
 \eqref{zttonefang11}, we obtain
\begin{eqnarray*}
	\bar{x}^*=\text{Proj}_{\mathcal{D}(\bar{\lambda}^*)}0.
\end{eqnarray*}

To prove the strong convergence of the primal trajectory, we need the following lemma which  was essentially established  in \cite[Lemma 5.1]{zhuhufang1}.

\begin{lemma}\label{lemma 5.1}
Suppose that $\bar{x}^*=\text{Proj}_{S}0$ and $\epsilon: [t_0, +\infty)\rightarrow [0,+\infty)$ is a  nonincreasing  and continuously differentiable function satisfying $\lim_{t\rightarrow+\infty}\epsilon(t)=0$. Let $(x,\lambda): [t_{0},+\infty)\rightarrow \mathcal{X}\times\mathcal{Y}$ be a solution of System \eqref{z2}.
Then, it holds
\begin{eqnarray*}
	\frac{\epsilon(t)}{2}(\|x(t)-x_{\epsilon(t)}\|^2+\|x_{\epsilon(t)}\|^2-\|\bar{x}^*\|^2)\leq\mathcal{L}_{\epsilon(t)}(x(t))-\mathcal{L}_{\epsilon(t)}(\bar{x}^*).
\end{eqnarray*}
\end{lemma}

\begin{theorem}\label{theoremztt5.1}

Suppose that Assumption \ref{AS-F} and  conditions \eqref{S-C}, \eqref{intbe} and \eqref{Slow} hold.
Let $(x,\lambda): [t_{0},+\infty)\rightarrow \mathcal{X}\times\mathcal{Y}$ be a solution of System \eqref{z2} and $\bar{x}^*=\text{Proj}_{S}0$. Then,
$$\liminf_{t\rightarrow+\infty}\|x(t)-\bar{x}^*\|=0.$$
Further, if there exists a constant $T>0$ such that the primal trajectory $\{x(t): t\geq T\}$ stays in either the open ball $B(0,\|\bar{x}^*\|)$ or its complement, then 
 $$\lim_{t\rightarrow+\infty}\|x(t)-\bar{x}^*\|=0.$$
\end{theorem}

\begin{proof}
Let $\bar{\lambda}^*\in\mathcal{Y}$ be a dual solution such that $(\bar{x}^*, \bar{\lambda}^*)\in\Omega$. It follows from \eqref{zs3} that 
\begin{eqnarray*}
	\frac{\mathcal{E}(t)}{\theta^2t^2}&=&\beta(t)\left(\mathcal{L}_{\rho}(x(t),\bar{\lambda}^*)-\mathcal{L}_{\rho}(\bar{x}^*,\bar{\lambda}^*)+\frac{\epsilon(t)}{2}\|x(t)\|^2\right)
	+\frac{1}{2}\left\|\frac{1}{\theta t}(x(t)-\bar{x}^*)+\dot{x}(t)\right\|^2\nonumber\\
	&&+\frac{1}{2}\left\|\frac{1}{\theta t}(\lambda(t)-\bar{\lambda}^*)+ \dot{\lambda}(t)\right\|^2+\frac{(\alpha-1)\theta-1}{2\theta^2 t^2}\|x(t)-\bar{x}^*\|^2\nonumber\\
	&&+\frac{(\alpha-1)\theta-1}{2\theta^2 t^2}\|\lambda(t)-\bar{\lambda}^*\|^2.
\end{eqnarray*}
Denote
\begin{eqnarray}\label{zr1}
	\widetilde{\mathcal{E}}(t)&=&\beta(t)\left(\mathcal{L}_{\epsilon(t)}(x(t))-\mathcal{L}_{\epsilon(t)}(\bar{x}^*)\right)
	+\frac{1}{2}\left\|\frac{1}{\theta t}(x(t)-\bar{x}^*)+\dot{x}(t)\right\|^2\nonumber\\
	&&+\frac{1}{2}\left\|\frac{1}{\theta t}(\lambda(t)-\bar{\lambda}^*)+ \dot{\lambda}(t)\right\|^2+\frac{(\alpha-1)\theta-1}{2\theta^2 t^2}\|x(t)-\bar{x}^*\|^2\nonumber\\
	&&+\frac{(\alpha-1)\theta-1}{2\theta^2 t^2}\|\lambda(t)-\bar{\lambda}^*\|^2.
\end{eqnarray}
Then, 
\begin{eqnarray*}
	\widetilde{\mathcal{E}}(t)=\frac{\mathcal{E}(t)}{\theta^2t^2}-\frac{\beta(t)\epsilon(t)}{2}\|\bar{x}^*\|^2.
\end{eqnarray*}
Using Assumption \ref{AS-F}, \eqref{S-C}, and \eqref{zr2},  we have 
\begin{eqnarray}\label{sunaf}
\frac{d}{dt}\left(\widetilde{\mathcal{E}}(t)\right)+\frac{2}{t}\widetilde{\mathcal{E}}(t)&=&\frac{\dot{\mathcal{E}}(t)}{\theta^2t^2}-\frac{\dot{\beta}(t)\epsilon(t)+\beta(t)\dot{\epsilon}(t)}{2}\|\bar{x}^*\|^2-\frac{\beta(t)\epsilon(t)}{t}\|\bar{x}^*\|^2\nonumber\\
&\leq&\left(\frac{(2\theta-1)\beta(t)+\theta t\dot{\beta}(t)}{2\theta t}\epsilon(t)+\frac{\beta(t)\dot{\epsilon}(t)}{2}\right)\left(\|x(t)\|^2-\|\bar{x}^*\|^2\right).
\end{eqnarray}

We will complete the proof by  analyzing separately the following three situations  according to the sign of $\|\bar{x}^*\|-\|x(t)\|$.

{\bf Case I:} There exists a large enough $T>0$ such that $\{x(t): t\geq T\}$ stays in the complement of the open ball $B(0,\|\bar{x}^*\|)$, which means 
\begin{eqnarray}\label{zf4}
	\|x(t)\|\geq\|\bar{x}^*\|, \quad \forall t\geq T.
\end{eqnarray}

Using  Assumption \ref{AS-F}, \eqref{S-C} again, we get
$$\frac{(2\theta-1)\beta(t)+\theta t\dot{\beta}(t)}{2\theta t}\epsilon(t)+\frac{\beta(t)\dot{\epsilon}(t)}{2}\leq0.$$  
This together with \eqref{sunaf} and \eqref{zf4} implies
\begin{eqnarray*}
	\frac{d}{dt}\left(\widetilde{\mathcal{E}}(t)\right)+\frac{2}{t}\widetilde{\mathcal{E}}(t)\leq0, \qquad \forall t\geq T.
\end{eqnarray*}    
Multiplying it by $t^{2}$, we obtain
\begin{eqnarray*}
\frac{d}{dt}\left(t^2\widetilde{\mathcal{E}}(t)\right)\leq0, \qquad \forall t\geq T.
\end{eqnarray*}    
Integrating the above inequality over $[T, t]$ and  then dividing by $t^2$, we have    
\begin{eqnarray*}
\widetilde{\mathcal{E}}(t)\leq \frac{T^2\widetilde{\mathcal{E}}(T)}{t^2}, \qquad \forall t\geq t_0, 
\end{eqnarray*} 
which together with \eqref{zr1} leads to
$$\beta(t)\left(\mathcal{L}_{\epsilon(t)}(x(t))-\mathcal{L}_{\epsilon(t)}(\bar{x}^*)\right)\leq\widetilde{\mathcal{E}}(t)\leq \frac{T^2\widetilde{\mathcal{E}}(T)}{t^2}.$$
By Lemma \ref{lemma 5.1}, \eqref{Slow}, and \eqref{zf3}, 
\begin{eqnarray*}
0\leq\|x(t)-x_{\epsilon(t)}\|^2\leq\frac{2T^2\widetilde{\mathcal{E}}(T)}{t^2\beta(t)\epsilon(t)}+\|\bar{x}^*\|^2-\|x_{\epsilon(t)}\|^2\to 0 \text{ as } t\to+\infty,
\end{eqnarray*} 
which means
$$\lim_{t\rightarrow+\infty}\|x(t)-\bar{x}^*\|=0.$$

{\bf Case II:} There exists a large enough $T>0$ such that $\{x(t): t\geq T\}$ is in the open ball $B(0,\|\bar{x}^*\|)$, i.e., 
\begin{eqnarray}\label{zr3}
	\|x(t)\|<\|\bar{x}^*\|, \quad \forall t\geq T.
\end{eqnarray}
Let $\bar{x}\in\mathcal{X}$ be a weak sequential cluster point of $(x(t))_{t\geq t_{0}}$. Then, there exists a subsequence $\{t_n\}_{n\in \mathbb{N}}\subseteq[T, +\infty)$ satisfying $t_n\to+\infty$ such that
\begin{eqnarray}\label{zr4}
	x(t_{n})\rightharpoonup\bar{x}   \quad as \quad n\rightarrow +\infty.
\end{eqnarray}
Then, we have 
\begin{eqnarray*}\label{zr5}
	\mathcal{L}_{\rho}(\bar{x},\bar{\lambda}^*)\leq\liminf_{n\rightarrow+\infty}\mathcal{L}_{\rho}(x(t_n),\bar{\lambda}^*).
\end{eqnarray*}
Further, according to Theorem \ref{ztt3.2}, we have
\begin{eqnarray*}
	\lim_{n\rightarrow+\infty}\mathcal{L}_{\rho}(x(t_n),\bar{\lambda}^*)-\mathcal{L}_{\rho}(\bar{x}^*,\bar{\lambda}^*)=0.
\end{eqnarray*}
Thus, we obtain 
\begin{eqnarray*}
	\mathcal{L}_{\rho}(\bar{x},\bar{\lambda}^*)\leq\mathcal{L}_{\rho}(\bar{x}^*,\bar{\lambda}^*)=\min_{x\in\mathcal{X}}\mathcal{L}_{\rho}(x,\bar{\lambda}^*)\leq\mathcal{L}_{\rho}(\bar{x},\bar{\lambda}^*))\leq\mathcal{L}_{\rho}(\bar{x}^*,\bar{\lambda}^*).
\end{eqnarray*}
This yields
$$\bar{x}\in\arg\min_{x\in\mathcal{X}}\mathcal{L}_{\rho}(x,\bar{\lambda}^*),$$
which together with \eqref{zttonefang11} implies $\bar{x}\in S$. 
Using \eqref{zr3}, we obtain
\begin{eqnarray*}
	\limsup_{n\rightarrow+\infty}\|x(t_{n})\|\leq\|\bar{x}^*\|.
\end{eqnarray*}
Since $\|\cdot\|$ is weakly lower semicontinuous, from \eqref{zr4} we get
\begin{eqnarray*}
	\|\bar{x}\|\leq\liminf_{t_{n}\rightarrow+\infty}\|x(t_{n})\|\leq\limsup_{n\rightarrow+\infty}\|x(t_{n})\|\le\|\bar{x}^*\|,
\end{eqnarray*}
which together with $\bar{x}\in S$ and $\bar{x}^*=\text{Proj}_{S}0$ implies that $\bar{x}=\bar{x}^*$.
Thus, $(x(t))_{t\geq t_0}$ enjoys a unique weak sequential cluster point $\bar{x}^*$. Consequently, 
\begin{equation}\label{xweak}
x(t)\rightharpoonup\bar{x}^*   \quad as \quad t\rightarrow +\infty
\end{equation}
and then
$$\|\bar{x}^*\|\leq\liminf_{t\rightarrow+\infty}\|x(t)\|.$$
On the other hand, from \eqref{zr3} we have
$$\limsup_{t\rightarrow+\infty}\|x(t)\|\leq\|\bar{x}^*\|.$$
As a result,
$$\lim_{t\rightarrow+\infty}\|x(t)\|=\|\bar{x}^*\|,$$
which, combined with  \eqref{xweak}, implies
$$\lim_{t\rightarrow+\infty}\|x(t)-\bar{x}^*\|=0.$$

{\bf Case III:} For any $T>t_0$, the trajectory $\{x(t): t\geq T\}$ does not remain in the  open ball $B(0,\|\bar{x}^*\|)$, nor its complent. In this case, there exists a subsequence $(t_n)_{n\in\mathbb{N}}\subseteq[t_0,+\infty)$ satisfying $t_n\rightarrow+\infty$ as $n\rightarrow+\infty$ such that
\begin{eqnarray*}\label{zr6}
	\|x(t_n)\|=\|\bar{x}^*\|, \quad \forall n\in\mathbb{N}.
\end{eqnarray*}
Let   $\tilde{x}\in\mathcal{X}$ be a weak sequential cluster point of $(x(t_n))_{n\in\mathbb{N}}$.  By arguments similar to Case II, we have
$$x(t_n)\rightharpoonup \bar{x}^*, \quad as \quad n\rightarrow+\infty.$$
Hence, we can obtain 
\begin{eqnarray*}
	\lim_{n\rightarrow+\infty}\|x(t_n)-\bar{x}^*\|=0,
\end{eqnarray*}
which implies
\begin{eqnarray*}
	\liminf_{t\rightarrow+\infty}\|x(t)-\bar{x}^*\|=0.
\end{eqnarray*}
\end{proof}

\begin{remark}\label{rmkstrong}
 Theorem \ref{ztt3.1} can be regarded as an extension of \cite[Theorem 4.1]{AttouchZH2018}, which establishes  the strong trajectory of  $\text{(AVD)}_{\epsilon}$  to the minimal norm solution of Problem \eqref{zfg1} under similar assumptions.   The results of Theorem \ref{ztt3.1} are also consistent with the ones reported in \cite[Theorem 4.2]{zhuhufang1} on $\text{(ZHF-AVD)}$.
\end{remark}


\section{Numerical experiments}

In this section, we illustrate the theoretical results on System \eqref{z2} by  two examples. All codes are run on a PC (with 2.20GHz Dual-Core Intel Core i7 and 16GB memory) under MATLAB Version R2017b. All dynamical systems are solved numerically by using the ode23 in MATLAB. 

\begin{example}\label{exam1}
	Consider the linear equality constrained quadratic programming problem	
	\begin{eqnarray}\label{LCPP}
		\min_{x}\quad \frac{1}{2}x^TQx+q^Tx, \quad\text{ s.t. } Ax=b, 
	\end{eqnarray}	
	where $A\in \mathbb{R}^{m\times n}$, $b\in\mathbb{R}^{m}$, $q\in\mathbb{R}^{n}$, and $Q\in\mathbb{R}^{n\times n}$ is a symmetric and positive definite matrix.
\end{example}
By this example, we test the behavior of System \eqref{z2} in  $|f(x(t))-f(x^*)|$,  $\|Ax(t)-b\|$, and  $\|(x(t), \lambda(t))-(x^*, \lambda^*)\|$, compared  with $\text{(Z\mbox{-}AVD)}$ considered in \cite{ZengXLandLeiJLandChenJ(2022),BNguyen2022}. 

Let $b$ be generated by the uniform distribution and $q$ be generated by the standard Gaussian distribution. Set $Q=H^TH+0.01Id$, where  $H\in\mathbb{R}^{n\times n}$ is generated by the standard Gaussian distribution and $Id$ denotes the identity matrix.
By using the Matlab function $quadprog$ with tolerance $10^{-15}$, we can obtain the unique  solution $x^*$ and the optimal value $f(x^*)$ of Problem \eqref{LCPP}. Further, take $\alpha=15$, $\theta=\frac{1}{13}$, $\rho=1$, $\epsilon(t)=\frac{1}{t^{r_1}}$ 
and $\beta(t)=t^{r_2}$ with $r_1-r_2>2$ in System \eqref{z2}, and take $\alpha=15$, $\theta=\frac{1}{13}$ and $\rho=1$ in $\text{(Z\mbox{-}AVD)}$.  Take the starting point $(x(1),\lambda(1),\dot{x}(1),\dot{\lambda}(1))=\mathbf{1}^{2(n+m)\times1}$ in  both System \eqref{z2} and $\text{(Z\mbox{-}AVD)}$.

First, let $m=20$, $n=50$, $r_1=4$ and $r_2=\{1, 1.5, 1.8\}$, and let $A\in\mathbb{R}^{m\times n}$ be generated by the standard Gaussian distribution. Figure \ref{fig:testfigwe} displays the behaviors of System \eqref{z2} and  $\text{(Z\mbox{-}AVD)}$ in  $|f(x(t))-f(x^*)|$ and $\|Ax(t)-b\|$.  The experiment result supports the theoretical result of Theorem \ref{ztt3.1}.

Second, let $A\in \mathbb{R}^{m\times n}$ with $m=n$ be an orthogonal matrix obtained by using QR decomposition for an invertible matrix $D\in \mathbb{R}^{n\times n}$, where 
$D\in \mathbb{R}^{n\times n} $ is generated by the standard Gaussian distribution. Then, the associated dual problem has a unique $\lambda^*\in\mathbb{R}^{m}$. In Figure \ref{fig:testfigwe2}, we plot the errors  $\|(x(t), \lambda(t))-(x^*, \lambda^*)\|$ of System \eqref{z2} under different choices of $m$ and $r_2$.

 As shown in Figure \ref{fig:testfigwe} and  Figure \ref{fig:testfigwe2},  System \eqref{z2} outperforms $\text{(PD\mbox{-}AVD)}$, and  the bigger the parameter $r_2$ is, the better  System \eqref{z2} performs.

\begin{figure*}[h]
	\centering
	{
		\begin{minipage}[t]{0.48\linewidth}
			\centering
			\includegraphics[width=2.75in]{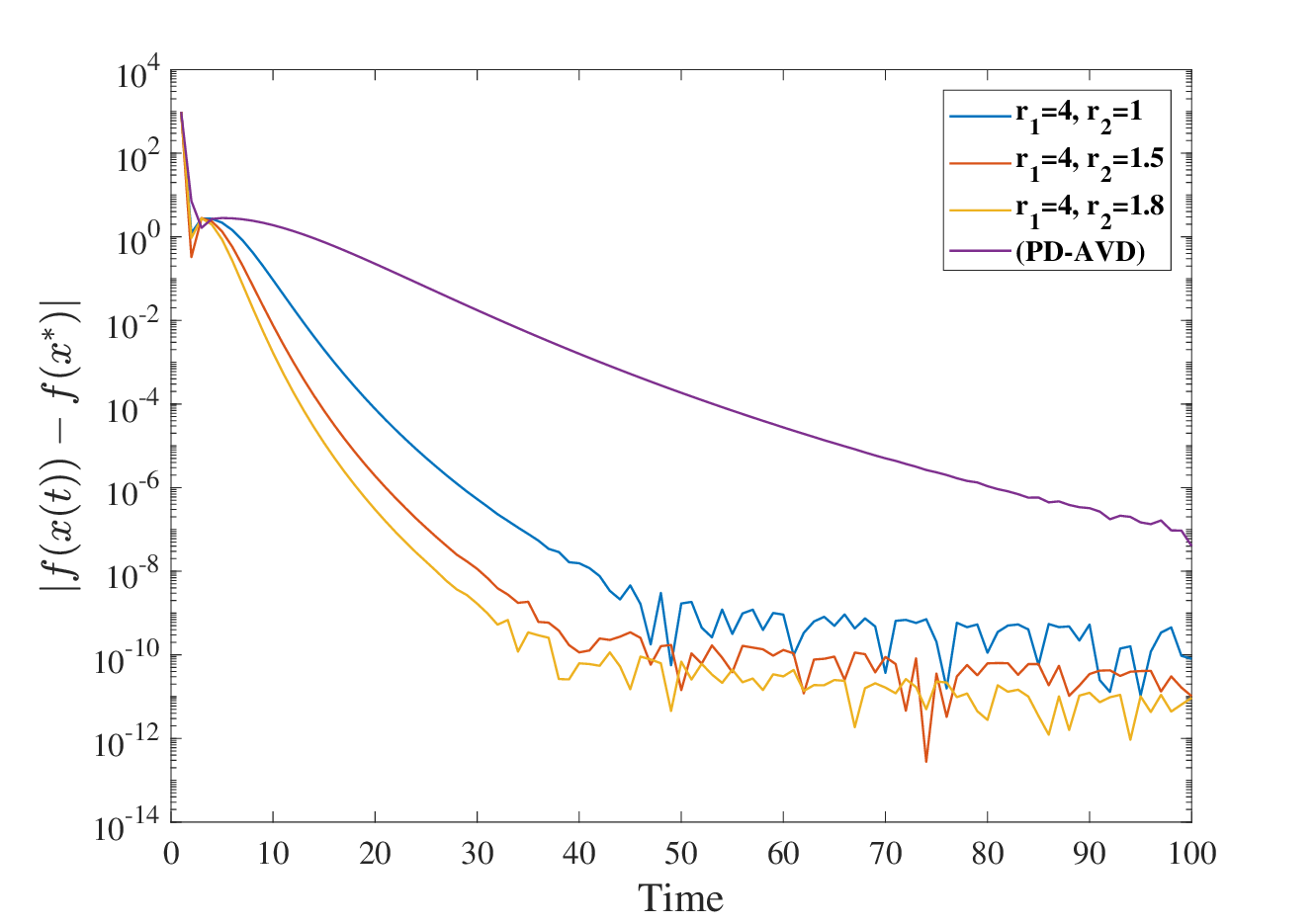}
		\end{minipage}%
	}
	{
		\begin{minipage}[t]{0.48\linewidth}
			\centering
			\includegraphics[width=2.75in]{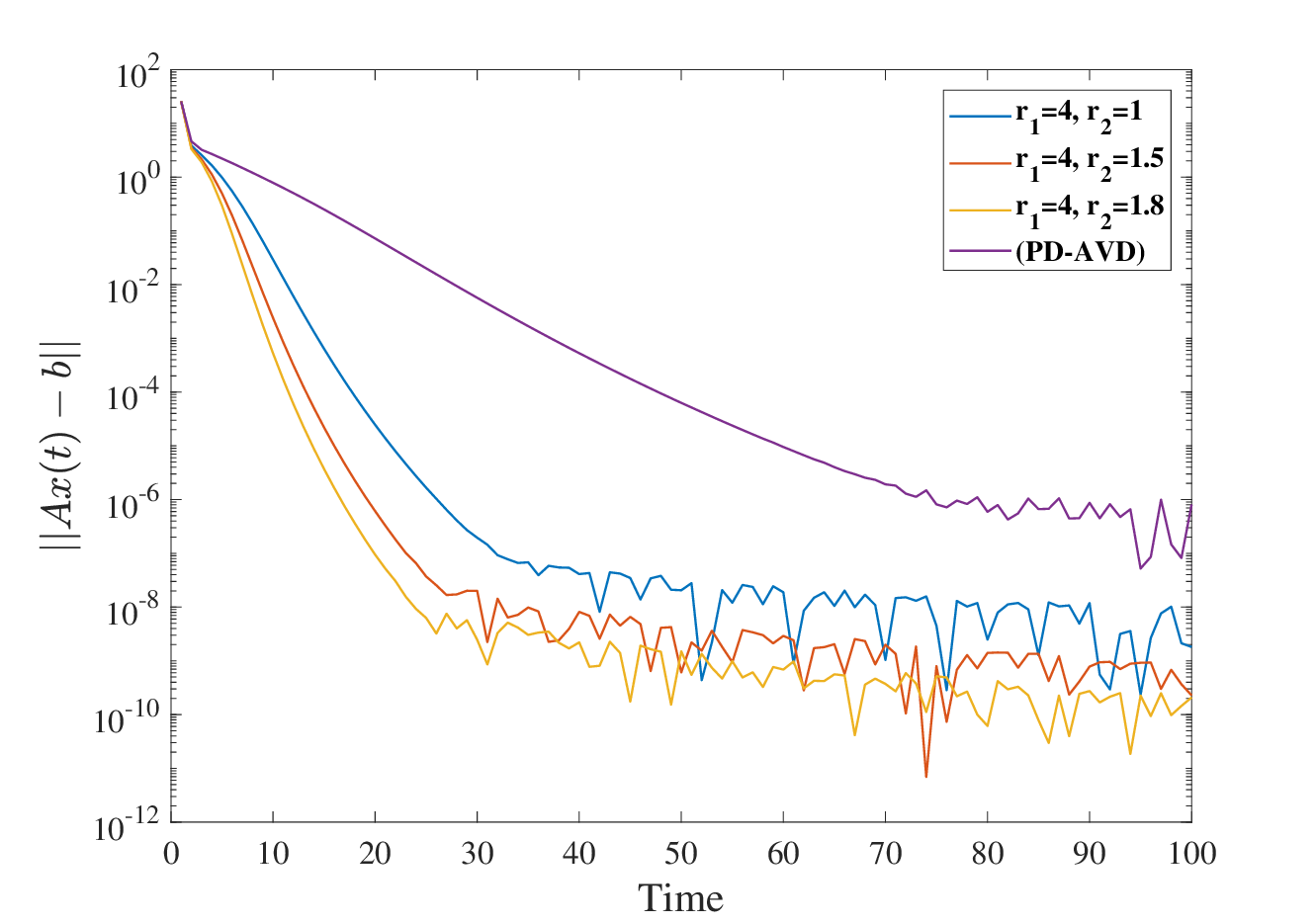}
		\end{minipage}%
	}
	\caption{Behaviors of   $|f(x(t))-f(x^*)|$ and $\|Ax(t)-b\|$ of System \eqref{z2} and (Z-AVD) in Example \ref{exam1}}\label{fig:testfigwe}
	\centering
\end{figure*}
\begin{figure*}[h]
	\centering
	\subfloat[m=n=10.]
	{
		\begin{minipage}[t]{0.48\linewidth}
			\centering
			\includegraphics[width=2.75in]{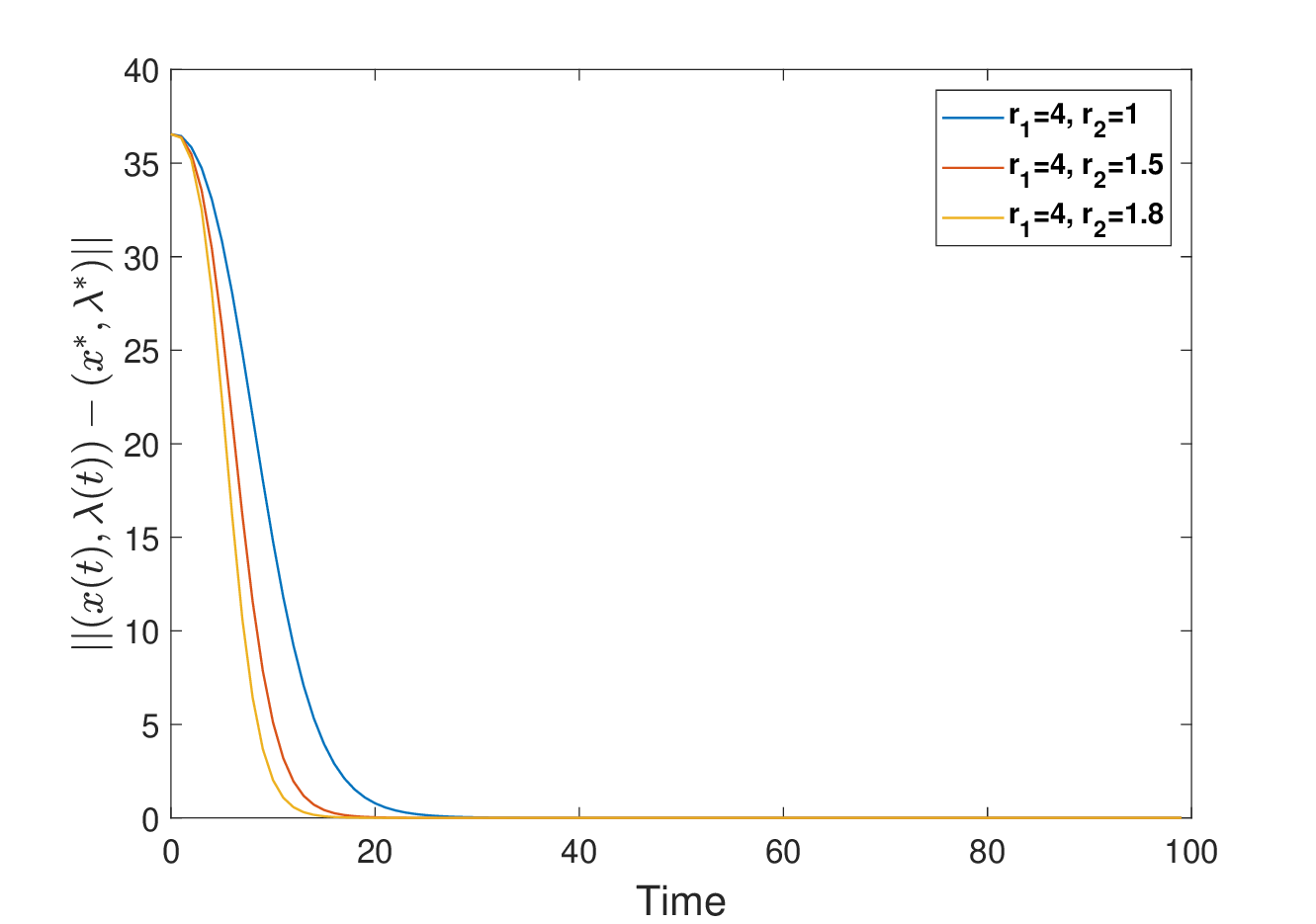}
		\end{minipage}%
	}
	\subfloat[m=n=50.]
	{
		\begin{minipage}[t]{0.48\linewidth}
			\centering
			\includegraphics[width=2.75in]{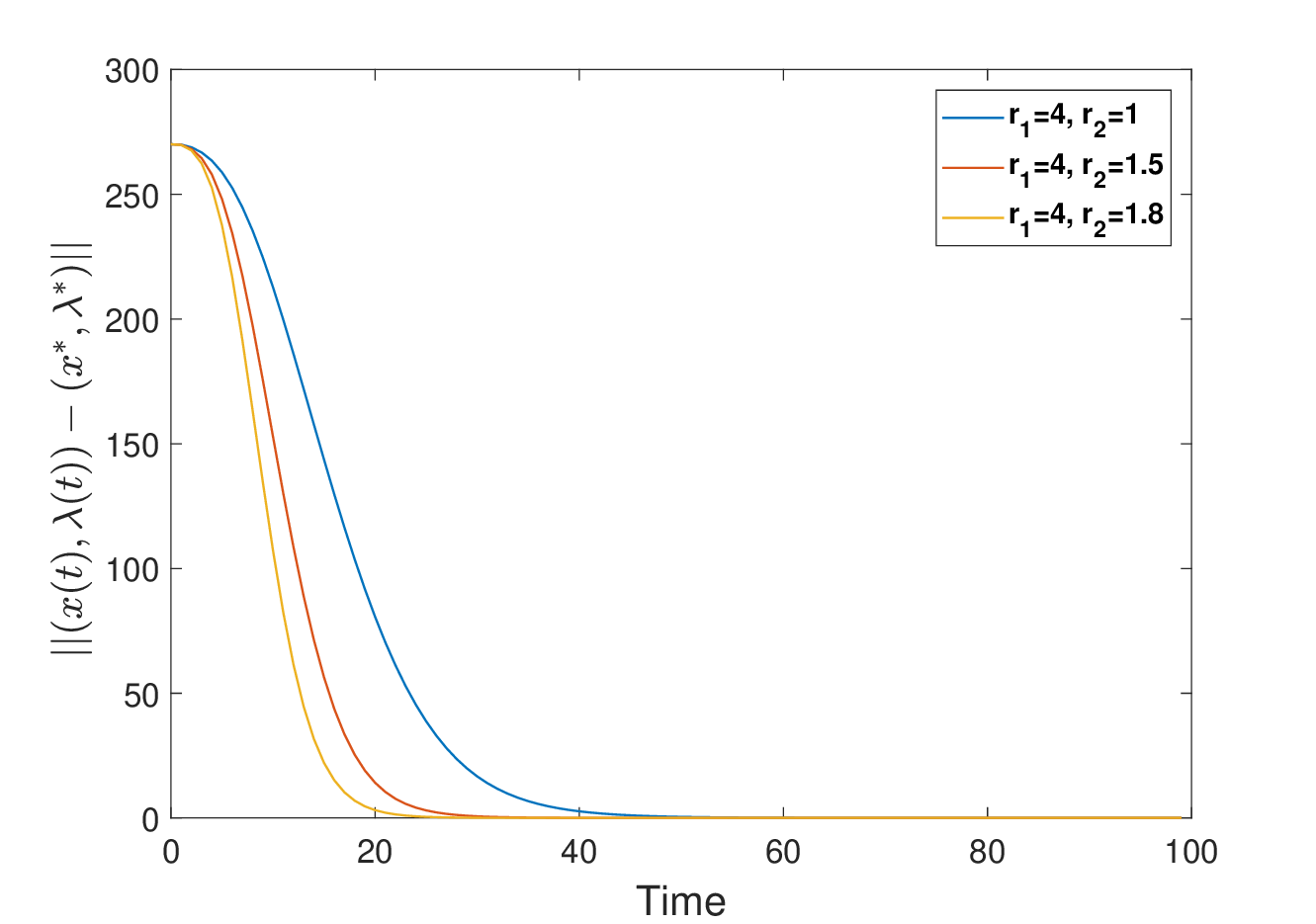}
		\end{minipage}%
	}
	\caption{Behavior of $\|(x(t), \lambda(t))-(x^*, \lambda^*)\|$ of  System \eqref{z2} in Example \ref{exam1} }\label{fig:testfigwe2}
	\centering
\end{figure*}

\begin{example}\label{exam2}
	Consider the equality constrained convex optimization problem
	\begin{eqnarray}\label{ztthufang2}
		\min_{x\in\mathbb{R}^3}\quad f(x)=(dx_{1}+ex_{2}+vx_{3})^2, \quad\text{ s.t. } Ax=b,
	\end{eqnarray}
	where $f:\mathbb{R}^3\rightarrow\mathbb{R}$, $ A=(d,-e,v)^T$, $b=0$, and $d,e,v\in\mathbb{R}\backslash\{0\}$.
\end{example}
This example was considered in \cite[Section 5]{zhuhufang1}, which was motivated by the example  in \cite[Section 4]{Laszlo-JDE} for an unconstrained optimization problem.  The solution set of Problem \eqref{ztthufang2} is $\{x^*=(x_{1},0,-\frac{d}{v}x_{1}): x_{1}\in\mathbb{R}\}$, the optimal value $f(x^*)=0$, and the minimal norm solution is $\bar{x}^*=(0,0,0)$.
Take the starting points $x(1)=(1,1,-1)$, $\lambda(1)=1$, $\dot{x}(1)=(1,1,1)$, $\dot{\lambda}(1)=1$.

Take $d=5,$ $e=1$,  $v=1$, $\rho=1$, $\alpha=13$, $\theta=\frac{1}{8}$, $\epsilon(t)=\frac{2.8}{t^{r_1}}$ with $r_1=\{1, 1.9, 2.4\}$ and $\beta(t)= t^{r_2}$ with $r_2=0.9$. Figure \ref{fig:testfig} presents the behaviors of  $\|x(t)-\bar{x}^*\|$ and $|f(x(t))-f(x^*)|$ along  the trajectory $(x(t))_{t\geq t_{0}}$ of System \eqref{z2}.
\begin{figure*}[htbp]
	\centering
	{
		\begin{minipage}[t]{0.48\linewidth}
			\centering
			\includegraphics[width=2.75in]{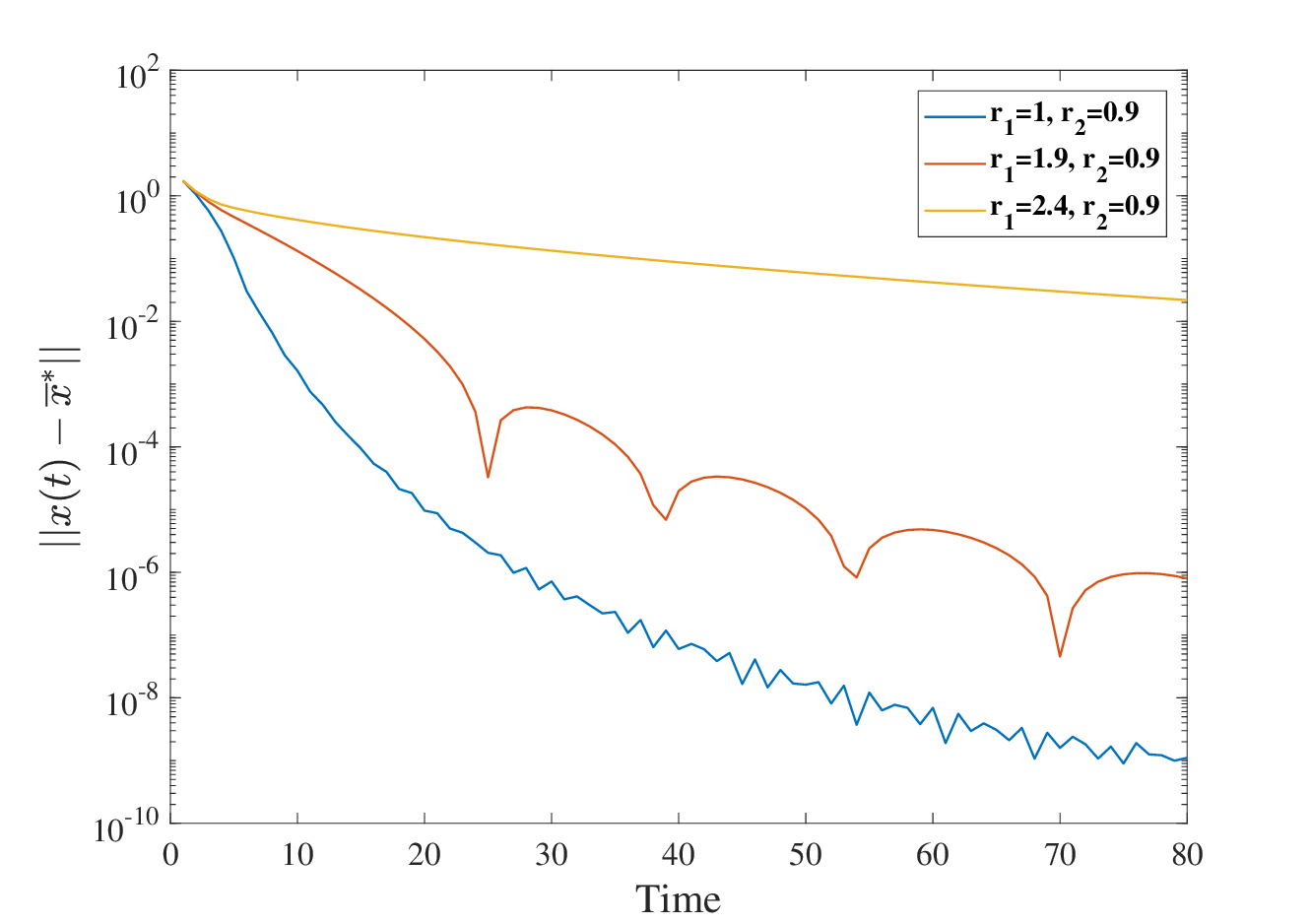}
		\end{minipage}%
	}
	{
		\begin{minipage}[t]{0.48\linewidth}
			\centering
			\includegraphics[width=2.75in]{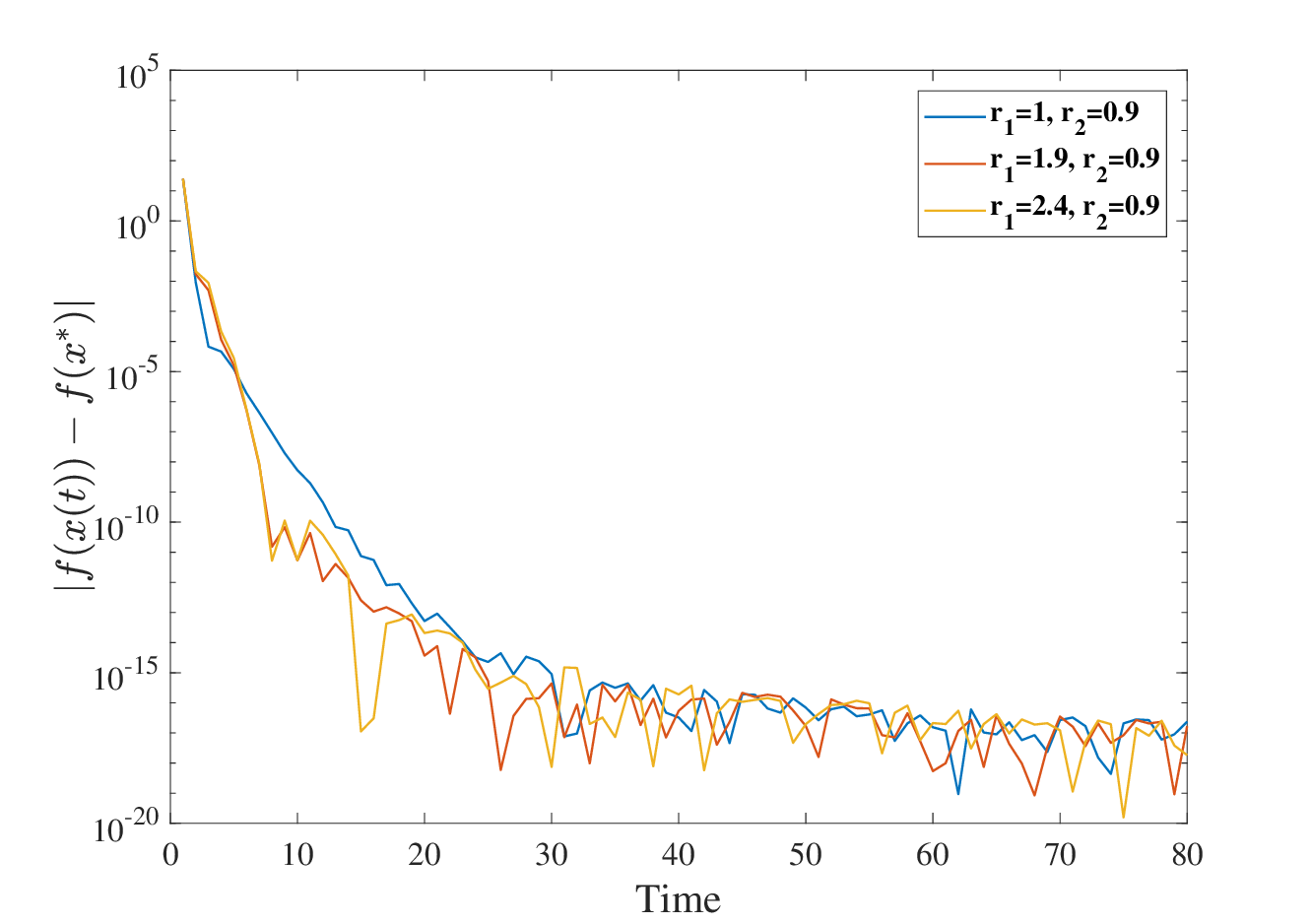}
		\end{minipage}%
	}
    \caption{Error analysis of System \eqref{z2}  with different parameters $r_1$ and $r_2$ in Example \ref{exam2}}\label{fig:testfig}
	\centering
\end{figure*}

As shown in Figure \ref{fig:testfig}, the primal trajectory $(x(t))_{t\geq t_{0}}$ of \eqref{z2} converges to the minimal norm solution $\bar{x}^*$. This supports the theoretical result of Theorem \ref{theoremztt5.1}.  Further,  System \eqref{z2} performs better in the error $\|x(t)-\bar{x}^*\|$ when $r_1-r_2$ is smaller, and the objective error $|f(x(t))-f(x^*)|$ is not very sensitive to the Tikhonov regularization parameter $r_1$. 

Remember that  System \eqref{z2} is just a Tikhonov regularized version of $\text{(HN-AVD)}$.   Next, we use Example \ref{exam1} to compare System \eqref{z2} with  $\text{(HN-AVD)}$ in the behaviors of the primal trajectory $x(t)$.  Take $\rho=1$, $\alpha=13$, $\theta=\frac{1}{8}$, $\beta(t)=t^{0.9}$ and $\epsilon(t)=\frac{2.8}{t}$ in System \eqref{z2} and take $\rho=1$, $\alpha=13$, $\theta=\frac{1}{8}$ and $\delta(t)=t^{0.9}$ in $\text{(HN-AVD)}$. Under different choices of $d$, $e$ and $v$, Figure \ref{fig:titstfig} displays the convergence of the primal trajectory $(x(t))_{t\geq t_0}$ generated by System \eqref{z2} to the minimal norm solution $\bar{x}^*=(0,0,0)$, while  the trajectory $(x(t))_{t\geq t_0}$ generated by $\text{(HN-AVD)}$ need not to converge to the minimal norm solution, which can be observed in Figure \ref{fig:titstfigst}.

\begin{figure*}[h]
	\centering
	\subfloat[d=5, e=1, v=1.]
	{
		\begin{minipage}[t]{0.48\linewidth}
			\centering
			\includegraphics[width=2.75in]{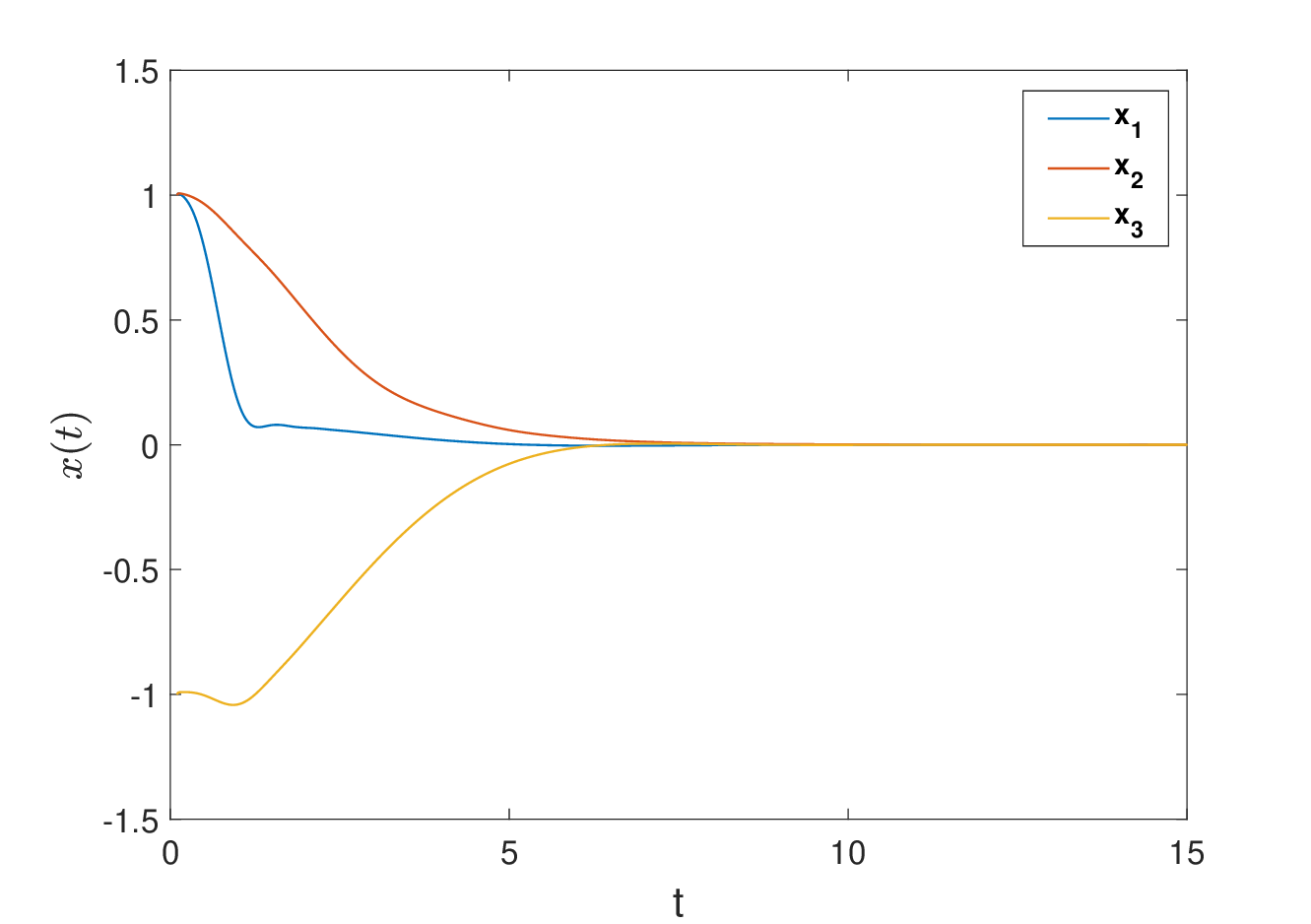}
		\end{minipage}%
	}
	\subfloat[d=120, e=5, v=25.]
	{
		\begin{minipage}[t]{0.48\linewidth}
			\centering
			\includegraphics[width=2.75in]{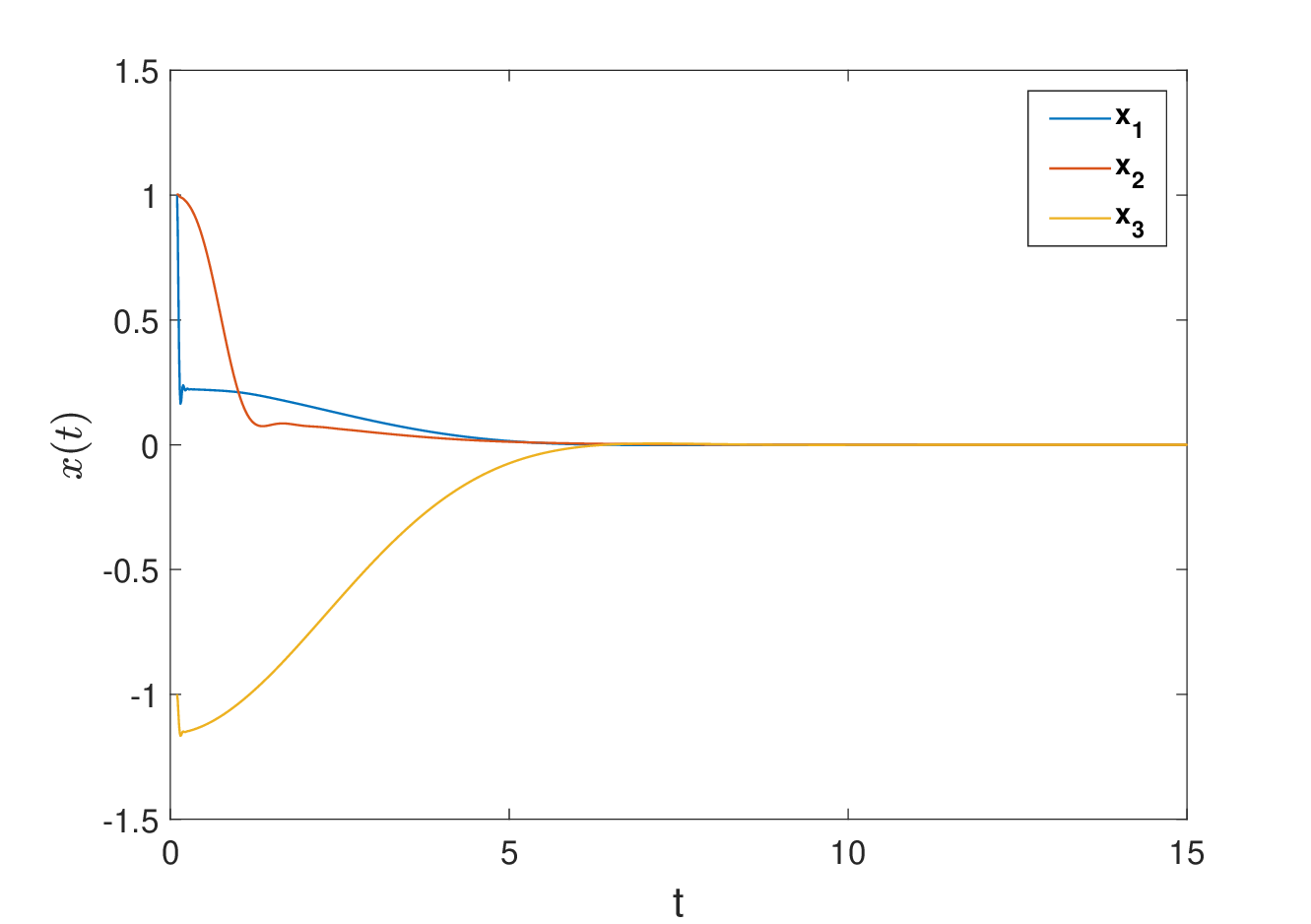}
		\end{minipage}%
	}
	\caption{The behavior of $(x(t))_{t\geq t_0}$ of System \eqref{z2}in Example \ref{exam2}}
	\label{fig:titstfig}
	\centering
\end{figure*}
\begin{figure*}[h]
	\centering
	\subfloat[d=5, e=1, v=1.]
	{
		\begin{minipage}[t]{0.48\linewidth}
			\centering
			\includegraphics[width=2.75in]{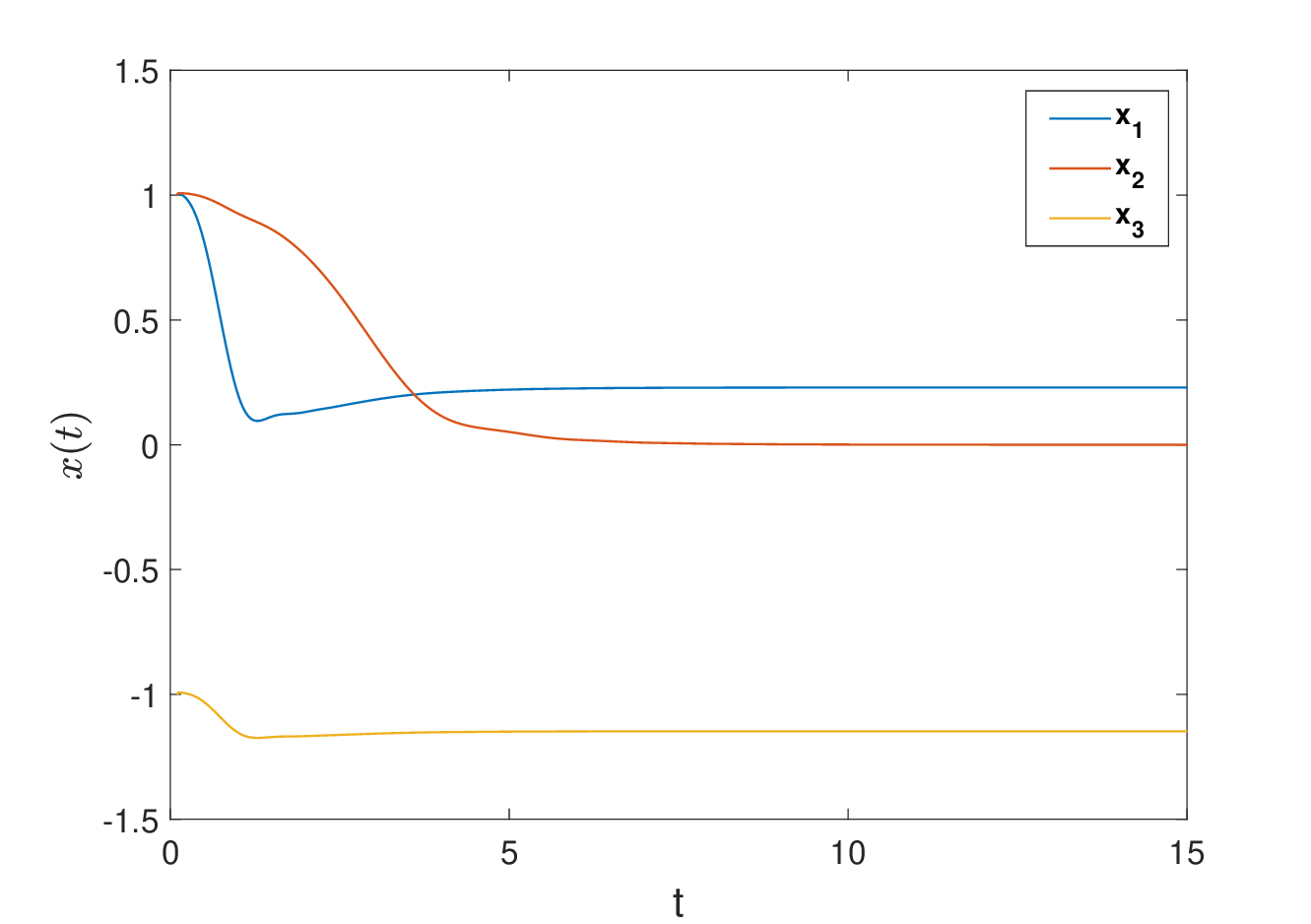}
		\end{minipage}%
	}
	\subfloat[d=120, e=5, v=25.]
	{
		\begin{minipage}[t]{0.48\linewidth}
			\centering
			\includegraphics[width=2.75in]{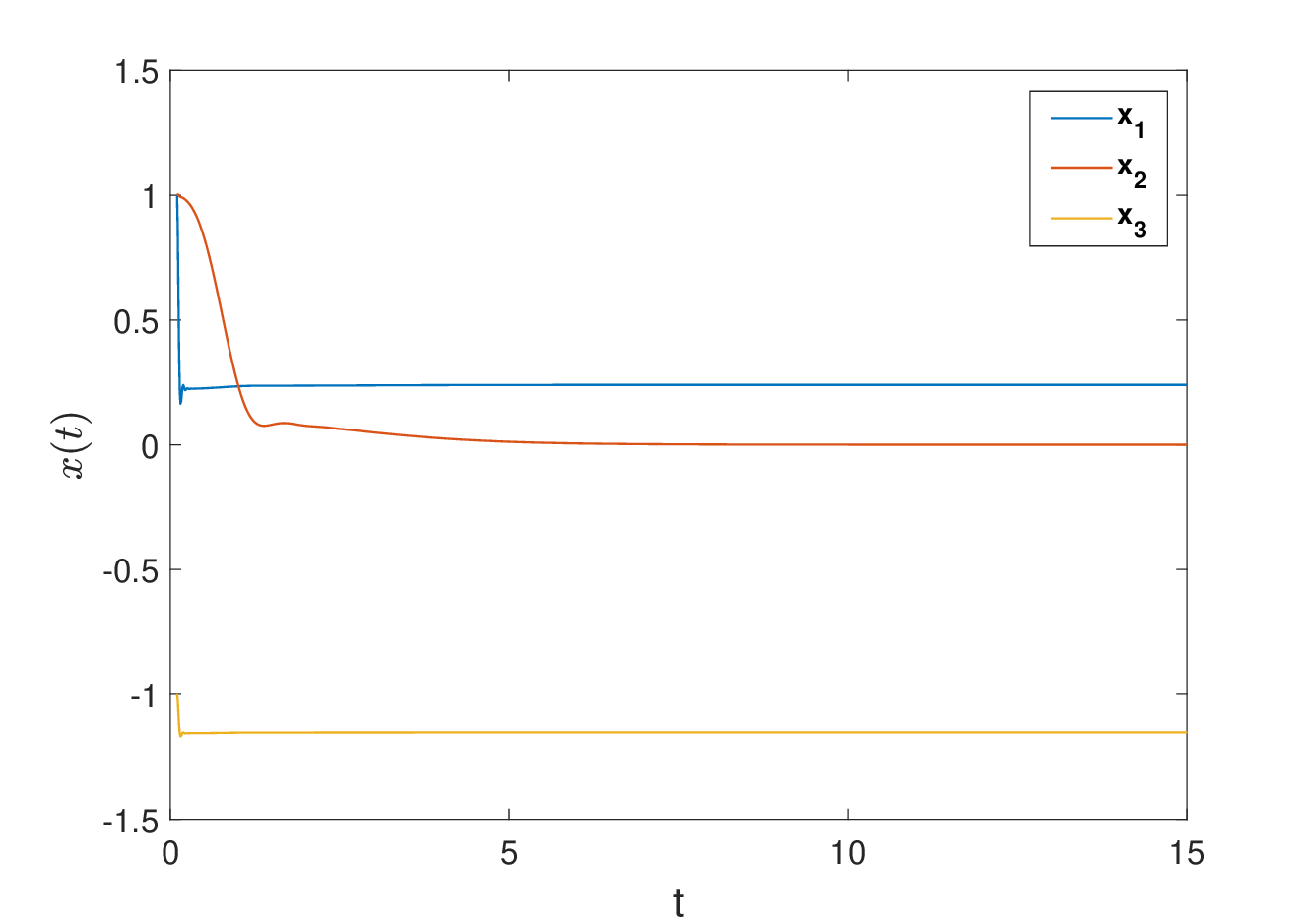}
		\end{minipage}%
	}
	\caption{The behavior of $(x(t))_{t\geq t_0}$ of $\text{(HN-AVD)}$ in Example \ref{exam2}}
	\label{fig:titstfigst}
	\centering
\end{figure*}

\appendix\section{Some auxiliary results}\label{append}

\begin{lemma}\label{lemma2.1.1}\cite{XuWen2021}
	Suppose that $\theta>0$ and $z\in \mathcal{X}$. Let $s: [t_0, +\infty)\rightarrow \mathcal{X}$ be a continuously differentiable function. If there exists $\widetilde{C}>0$ such that
	$$\frac{1}{2}\|s(t)-z+\theta t \dot{s}(t)\|^2\leq \widetilde{C}, \quad \forall t\geq t_0,$$
	then $(s(t))_{t\geq t_0}$ is bounded.
\end{lemma}

\begin{lemma}\label{lemma2.1.2}\cite[Lemma 6]{HeHFiietal(2022)}
	Suppose that  $g:[t_{0},+\infty)\rightarrow\mathbb{R}^n$  is a continuously differentiable function and that $a : [t_{0},+\infty)\rightarrow[0,+\infty)$ is a continuously differentiable function, where $t_{0}>0$.  If there exists  a constant $C\geq0$ such that
	$$\|g(t)+\int_{t_{0}}^{t}a(s)g(s)ds\|\leq C, \quad \forall t\geq t_{0},$$
	then
	$$\sup_{t\geq t_{0}}\|g(t)\|<+\infty.$$
\end{lemma}

\begin{lemma}\label{lemma2.1.3}\cite[Lemma A.1]{BNguyen2022}
	Assume that $0<\delta\leq p\leq+\infty$ and $K: [\delta, +\infty)\rightarrow[0, +\infty)$ is a continuous function. For any $\alpha>1$, it holds 
	$$\int_{\delta}^{p}\frac{1}{t^{\alpha}}\left(\int_{\delta}^{t}s^{\alpha-1}K(s)ds\right)dt\leq\frac{1}{\alpha-1}\int_{\delta}^{p}K(t)dt.$$
	If $p=+\infty$, then equality holds.
\end{lemma}

\begin{lemma}\label{lemma2.1.5}\cite[Lemma 5.2]{AbbasAttouchS2014}
	Let $\delta>0$, $1\leq c<+\infty$ and $1\leq d\leq+\infty$. Assume that $F\in \mathbb{L}^c([\delta, +\infty))$ is a locally absolutely continuous nonnegative function, $V\in \mathbb{L}^d([\delta, +\infty))$ and for almost every $t\geq \delta$
	$$\frac{d}{dt}F(t)\leq V(t).$$
	Then $\lim_{t\rightarrow+\infty}F(t)=0.$
\end{lemma}

\begin{lemma}\label{lemma2.1.6}\cite{Opialweak1967}
	Let $S$ be a nonempty subset of $\mathcal{X}$ and $y: [t_0, +\infty)\rightarrow \mathcal{X}$. Suppose that
	\begin{itemize}
		\item[(i)] for any $y^*\in S$, $\lim_{t\rightarrow+\infty}\|y(t)-y^*\|$ exists;
		\item[(ii)] every weak sequential cluster point of the trajectory $y(t)$ as $t\rightarrow+\infty$ belongs to $S$.
	\end{itemize}
	Then, the trajectory $x(t)$ converges weakly to a point in $S$ as $t\rightarrow+\infty$.
\end{lemma}

\begin{lemma}\label{lemma2.2.7}\cite[Lemma A.3]{AttouchZH2018}
	Suppose that $\delta>0$, $\phi\in L^{1}([\delta, +\infty))$ is a nonnegative and continuous function, and $\psi :[\delta, +\infty)\rightarrow(0, +\infty)$ is a nondecreasing function such that $\lim_{t\rightarrow+\infty}\psi(t)=+\infty$. Then,
	$$\lim_{t\rightarrow+\infty}\frac{1}{\psi(t)}\int_{\delta}^{t}\psi(s)\phi(s)ds=0.$$
\end{lemma}

\end{document}